\documentclass[a4paper]{amsart}

\usepackage[active]{srcltx}
\usepackage[all]{xy}
\usepackage{subfig}
\usepackage{hyperref}
\usepackage{mathrsfs}

\usepackage{esint}
\usepackage{amsmath}
\usepackage{amssymb,latexsym}
\usepackage{enumitem}
\usepackage{mathrsfs}
\usepackage{graphics}
\usepackage{latexsym}
\usepackage{psfrag}
\usepackage{import}
\usepackage{verbatim}
\usepackage{graphicx}
\usepackage[usenames]{color}
\usepackage{pifont,marvosym}

\theoremstyle{plain}
\newtheorem{lemma}{Lemma}[section]
\newtheorem{theorem}[lemma]{Theorem}
\newtheorem{proposition}[lemma]{Proposition}
\newtheorem{corollary}[lemma]{Corollary}

\theoremstyle{definition}

\newtheorem{definition}[lemma]{Definition}
\newtheorem{remark}[lemma]{Remark}

\numberwithin{equation}{section}

\newcommand{\be}{\begin{equation}}
\newcommand{\ee}{\end{equation}}
\DeclareMathOperator{\A}{\mathcal{A}}

\newcommand{\dom}{\textrm{Dom\,}}

\newcommand{\R}{\mathbb{R}}

\newcommand{\N}{\mathbb{N}}
\newcommand{\Q}{\mathbb{Q}}

\newcommand{\supp}{\text{\rm supp}}

\newcommand{\Z}{\mathbb{Z}}

\newcommand{\gr}{\textrm{graph}}

\newcommand{\ve}{\varepsilon}

\newcommand{\erre}{\mathbb{R}}
\newcommand{\cI}{\mathcal{I}}
\newcommand{\enne}{\mathbb{N}}
\newcommand{\HH}{\mathbb{H}}
\newcommand{\f}{\varphi}

\newcommand{\T}{\mathcal{T}}
\renewcommand{\r}{\varrho}

\renewcommand{\L}{\mathcal{L}}
\newcommand{\RCD}{\mathsf{RCD}}

\newcommand{\CD}{\mathsf{CD}}
\newcommand{\Geo}{{\rm Geo}}
\newcommand{\MCP}{\mathsf{MCP}}
\newcommand{\mm}{\mathfrak m}
\newcommand{\qq}{\mathfrak q}
\newcommand{\QQ}{\mathfrak Q}

\newcommand{\sfd}{\mathsf d}
\newcommand{\Opt}{\mathrm{OptGeo}}

\begin{document}

\author{Fabio Cavalletti}\thanks{Universit\`a degli Studi di Pavia, Dipartimento di Matematica, email: fabio.cavalletti@unipv.it} 
\title[An Overview of $L^{1}$ optimal transportation]{An Overview of $L^{1}$ optimal transportation \\ on metric measure spaces}


\bibliographystyle{plain}

\begin{abstract}
The scope of this note is to make a self-contained survey of the recent developments and achievements of the theory of $L^{1}$-Optimal Transportation on metric measure spaces. 
Among the results proved in the recent papers \cite{CM1,CM2} where the author, together with A. Mondino, 
proved a series of sharp (and in some cases rigid) geometric and functional inequalities in the setting of 
metric measure spaces enjoying a weak form of Ricci curvature lower bound, we 
review the proof of the L\'evy-Gromov isoperimetric inequality.
\end{abstract}

\maketitle


\section{Introduction}

The scope of this note is to make a self-contained survey of the recent developments and achievements of the theory of $L^{1}$-Optimal Transportation on metric measure spaces. 
We will focus on the general scheme adopted in the recent papers \cite{CM1,CM2} where the author, together with A. Mondino, 
proved a series of sharp (and in some cases even rigid and stable) geometric and functional inequalities in the setting of 
metric measure spaces enjoying a weak form of Ricci curvature lower bound. 
Roughly the general scheme consists in reducing the initial problem to a family of easier one-dimensional problems; 
as it is probably the most relevant result obtained with this technique, we will review in detail how to proceed to obtain the L\'evy-Gromov isoperimetric inequality for metric measure spaces 
verifying the Riemmanian Curvature Dimension condition (or, more generally, essentially non-branching 
metric measure spaces verifying the Curvature Dimension condition). 

In \cite{biacava:streconv, cava:MongeRCD} a fine analysis of the Monge problem in the metric setting was done treating, with a different perspective, 
similar questions whose answers were later used also in \cite{CM1,CM2}. 
We therefore believe the Monge problem and V.N. Sudakov's approach to it (see \cite{sudakov}) is a good starting point for our review and to see 
how $L^{1}$-Optimal Transportation naturally yields a reduction of the problem to a family of one-dimensional problems. 

It is worth stressing that the dimensional reduction proposed by V.N. Sudakov to solve the Monge problem is only one of the strategy to attack the problem.
Monge problem has a long story and many different authors contributed to obtain solutions in different frameworks with different approaches; 
here we only mention that the first existence result for the Monge problem was independently obtained in \cite{caffa:Monge} and in \cite{trudi:Monge}. 
We also mention the subsequent generalizations obtained in \cite{Ambrosio:Monge,AKP,feldMccann} and we refer to the monograph \cite{Vil:topics} 
for a more complete list of results.

\medskip
\subsection{Monge problem}
The original problem posed by Monge in 1781 can be restated in modern language as follows: 
given two Borel probability measures $\mu_{0}$ and $\mu_{1}$ over $\R^{d}$, called marginal measures,  
find the optimal manner of transporting $\mu_{0}$ to $\mu_{1}$; the transportation of $\mu_{0}$ to $\mu_{1}$ is understood as a map $T : \R^{d} \to \R^{d}$  
assigning to each particle $x$ a final position $T(x)$ fulfilling the following compatibility condition
\begin{equation}\label{E:transportmap}
 \qquad T_{\sharp} \, \mu_{0} = \mu_{1}, \qquad \textrm{i.e. }\quad  \mu_{0}(T^{-1}(A)) = \mu_{1}(A), \quad \forall \, A \ \textrm{Borel set};
\end{equation}
any map $T$ verifying the previous condition will be called a transport map.
The optimality requirement is stated as follows:
\begin{equation}\label{E:MongeEuclid}
\int_{\R^{d}} |T(x) - x| \, \mu_{0}(dx) \leq \int_{\R^{d}} |\hat T(x) -x | \, \mu_{0}(dx), 
\end{equation}
for any other $\hat T$ transport map. 
In proving the existence of a minimizer, 
the first difficulty appears studying the domain of the minimization, that is the set of maps $T$ verifying \eqref{E:transportmap}. Suppose 
$\mu_{0} = f_{0} \L^{d}$ and $\mu_{1} = f_{1} \L^{d}$ where $\L^{d}$ denotes the $d$-dimensional Lebesgue measure; 
a smooth injective map $T$ is then a transport map if and only if 
$$
f_{1} (T(x)) |\det (DT)(x) | = f_{0}(x), \qquad \mu_{0}\textrm{-a.e.} \ x \in \R^{d},
$$
showing a strong non-linearity of the constrain. 
The first big leap in optimal transportation theory was  achieved by Kantorovich considering a suitable relaxation of the problem: 
associate to each transport map the probability measure $(Id,T)_{\sharp} \mu_{0}$ over $\R^{d}\times \R^{d}$ and introduce the set 
of \emph{transport plans}
$$
\Pi(\mu_{0},\mu_{1}) : = \left\{ \pi \in \mathcal{P}(\R^{d}\times \R^{d}) \colon P_{1\,\sharp} \pi = \mu_{0},\  P_{2\,\sharp} \pi = \mu_{1} \right\}; 
$$
where $P_{i} : \R^{d} \times \R^{d} \to \R^{d}$ is the projection on the $i$-th component, with $i =1,2$.  
By definition $(Id,T)_{\sharp} \mu_{0} \in \Pi(\mu_{0},\mu_{1})$ and
$$
\int_{\R^{d}} |T(x) - x| \, \mu_{0}(dx) = \int_{\R^{d}\times \R^{d}} |x-y| \, \left( (Id, T)_{\sharp} \mu_{0}\right)(dxdy);
$$
then it is natural to consider the minimization of the following functional (called Monge-Kantorovich minimization problem)
\begin{equation}\label{E:MK}
\Pi(\mu_{0},\mu_{1}) \ni \pi \longmapsto \mathcal{I}(\pi)  : = \int_{\R^{d}\times\R^{d}} |x-y| \, \pi(dxdy).
\end{equation}
The big advantage being now that $\Pi(\mu_{0},\mu_{1})$ is a convex subset of $\mathcal{P}(\R^{d}\times \R^{d})$ and it is compact 
with respect to the weak topology.
Since the functional $\mathcal{I}$ is linear, the existence of a minimizer follows straightforwardly.
Then a strategy to obtain a solution of the original Monge problem is to start from an optimal transport plan $\pi$ and prove 
that it is indeed concentrated on the graph of a Borel map $T$; the latter is equivalent to $\pi = (Id,T)_{\sharp} \mu_{0}$. 

To run this program one needs to deduce from optimality some condition on the geometry of the support of the transport plan. 
This was again obtained by Kantorovich introducing a dual formulation of \eqref{E:MK} and finding out that for any probability measures 
$\mu_{0}$ and $\mu_{1}$ with finite first moment, there exists a $1$-Lipschitz function $\f : \R^{d} \to \R$ such that 
$$
\Pi(\mu_{0},\mu_{1}) \ni \pi \ \textrm{is optimal} \quad \iff \quad \pi \big(\{ (x,y) \in \R^{2d} \colon \f(x) - \f(y) = |x-y| \} \big) = 1.
$$
At this point one needs to focus on the structure of the set
\begin{equation}\label{E:transportproduct}
\Gamma : = \big\{ (x,y) \in \R^{2d} \colon \f(x) - \f(y) = |x-y| \big\}.
\end{equation}

\begin{definition}
A set $\Lambda \subset \R^{2d}$ is $|\cdot|$-cyclically monotone if and only if for any finite subset of $\Lambda$, 
$\{ (x_{1},y_{1}), \dots, (x_{N},y_{N})\} \subset \Lambda$  it holds
$$
\sum_{1\leq i\leq N} |x_{i} - y_{i}| \leq \sum_{1 \leq i\leq N} |x_{i} - y_{i+1}|, 
$$
where $y_{N+1} : = y_{1}$.
\end{definition}

Almost by definition, the set $\Gamma$ is $|\cdot|$-cyclically monotone 
and whenever $(x,y) \in \Gamma$ considering $z_{t} : = (1-t) x + t y$ with $t \in [0,1]$ it holds that $(z_{s},z_{t}) \in \Gamma$, 
for any $s \leq t$. In particular this suggests that $\Gamma$ produces a family of disjoint lines of $\R^{d}$ along where the optimal transportation should move. 
This can be made rigorous considering the following ``relation'' between points: 
a point $x$ is in relation with $y$ if, using optimal geodesics selected by the above optimal transport problem, one can travel from $x$ to $y$ or viceversa. 
That is, consider $R : = \Gamma \cup \Gamma^{-1}$ and define $x \sim y$ if and only if $(x,y) \in R$.
Then $\R^{d}$ will be decomposed (up to a set of Lebesgue-measure zero) as $\T \cup Z$ 
where $\T$ will be called the \emph{transport set} and $Z$ the set of points not moved by the optimal transportation problem.
The important property of $\T$ being that
$$
\T = \bigcup_{q\in Q} X_{q}, \qquad X_{q} \textrm{ straight line}, \qquad X_{q} \cap X_{q'} = \emptyset, \quad \textrm{if } q \neq q'.
$$
Here $Q$ is a set of indices; a convenient way to index a straight line $X_{q}$ is to select an element of $X_{q}$ and call it, with an abuse of notation, $q$. 
With this choice the set $Q$ can be understood as a subset of $\R^{d}$.
Once a partition of the space is given, one obtains via Disintegration Theorem a corresponding decomposition of marginal measures: 
$$
\mu_{0} = \int_{Q} \mu_{0\,q} \, \qq(dq), \qquad \mu_{1} = \int_{Q} \mu_{1\, q} \, \qq(dq);
$$
where $\qq$ is a Borel probability measure over the set of indices $Q \subset \R^{d}$. If $Q$ enjoys a measurability condition (see Theorem \ref{T:disintr} for details), 
the conditional measures $\mu_{0\,q}$ and $\mu_{1\,q}$ are concentrated on the straight line with index $q$, i.e. $\mu_{0\,q} (X_{q}) =  \mu_{1\,q} (X_{q})=  1$,  for $\qq$-a.e. $q \in Q$.

Then a classic way to construct an optimal transport maps is to 

\begin{itemize}
\item[-] consider $T_{q}$ the monotone rearrangement along $X_{q}$ of $\mu_{0\, q}$ to $\mu_{1\,q}$;
\item[-] define the transport map $T$ as $T_{q}$ on each $X_{q}$. 
\end{itemize}
The map $T$ will be then an optimal transport map moving $\mu_{0}$ to $\mu_{1}$; it is indeed easy to check that 
$(Id,T)_{\sharp}\mu_{0} \in \Pi(\mu_{0},\mu_{1})$ and $(x,T(x)) \in \Gamma$ for $\mu_{0}$-a.e. $x$.

\medskip

So the original Monge problem has been reduced to the following family of one-dimensional problems: 
for each $q \in Q$ find a minimizer of the following functional 
$$
\Pi(\mu_{0\,q},\mu_{1\,q}) \ni \pi \longmapsto \mathcal{I}(\pi)  : = \int_{X_{q}\times X_{q}} |x-y| \,\pi(dxdy),
$$
that is concentrated on the graph of a Borel function.
As $X_{q}$ is isometric to the real line, whenever $\mu_{0\,q}$ does not contain any atom (i.e $\mu_{0\,q} (x) = 0$, for all $x \in X_{q}$), 
the monotone rearrangement $T_{q}$ exists and the existence of an optimal transport map $T$ constructed as before follows.
The existence of a solution has been reduced therefore to a regularity property of the disintegration of $\mu_{0}$. 

As already stressed before, this approach to the Monge problem, mainly due to V.N. Sudakov, was proposed in \cite{sudakov} 
and was later completed in the subsequent papers \cite{caffa:Monge} and in \cite{trudi:Monge}. 
See also \cite{caravenna} for a complete Sudakov approach to Monge problem when the Euclidean distance is replaced by any strictly convex norm and 
\cite{biadaneri} where any norm is considered. 
In all these papers, assuming $\mu_{0}$ to be absolutely continuous with respect to $\L^{d}$ give the sufficient regularity to solve the problem.

\medskip

The Monge problem can be actually stated, and solved, in a much more general framework. 
Given indeed two Borel probability measures $\mu_{0}$ and $\mu_{1}$ over a complete and 
separable metric space $(X,\sfd)$, the notion of transportation map perfectly makes 
sense and the optimality condition \eqref{E:MongeEuclid} can be naturally formulated using the distance $\sfd$ as a cost function instead of the Euclidean norm:
\begin{equation}\label{E:Mongemetric}
\int_{\R^{d}} \sfd(T(x), x) \, \mu_{0}(dx) \leq \int_{\R^{d}} \sfd(\hat T(x),x) \, \mu_{0}(dx). 
\end{equation}
The problem can be relaxed to obtain a transport plan $\pi$ solution of the corresponding Monge-Kantorovich minimization problem. 
Also the Kantorovich duality applies yielding the existence of a $1$-Lipschitz function $\f : X \to \R$ such that 
$$
\Pi(\mu_{0},\mu_{1}) \ni \pi \ \textrm{is optimal} \quad \iff \quad \pi \big( \Gamma \big) = 1,
$$
where $\Gamma := \{ (x,y) \in X\times X \colon \f(x) - \f(y) = \sfd(x,y) \}$ is $\sfd$-cyclically monotone.  \\
All the strategy proposed for the Euclidean problem can be adopted: produce a decomposition of $X$ as
$\T \cup Z$ where $Z$ is the set of points not moved by the optimal transportation problem and 
$\T$ is the transport set and it is partitioned, up to a set of measure zero, by a family of geodesics $\{ X_{q} \}_{q \in Q}$; 
via Disintegration Theorem one obtains as before a reduction of the Monge problem to a family of one-dimensional problems
$$
\Pi(\mu_{0\,q},\mu_{1\,q}) \ni \pi \longmapsto \mathcal{I}(\pi)  : = \int_{X_{q}\times X_{q}} \sfd(x,y) \,\pi(dxdy).
$$
Therefore, since $X_{q}$ with distance $\sfd$ is isometric to an interval of the real line with Euclidean distance, 
the problem is reduced to proving that for $\qq$-a.e. $q \in Q$ the conditional measure $\mu_{0\,q}$ does not have any atoms.

Clearly in showing such a result, besides the regularity of $\mu_{0}$ itself, the regularity of the ambient space $X$ does play a crucial role.
In particular, together with the localization of the Monge problem to $X_{q}$, it should come a localization of the regularity of the space.  
This is the case when the metric space $(X,\sfd)$ is endowed with a reference probability measure $\mm$ and the resulting 
metric measure space $(X,\sfd,\mm)$ verifies a weak Ricci curvature lower bound.  

In \cite{biacava:streconv} we in fact  observed that if $(X,\sfd,\mm)$ verifies the so-called measure contraction property $\MCP$, then
for $\qq$-a.e. $q \in Q$ the one-dimensional metric measure space $(X_{q},\sfd,\mm_{q})$ verifies $\MCP$ as well, 
where $\mm_{q}$ is the conditional measure of $\mm$ with respect to the family of geodesics $\{ X_{q}\}_{q \in Q}$.
Now the assumption $\mu_{0}\ll \mm$ is sufficient to solve the Monge problem.
It is worth mentioning that \cite{biacava:streconv} was the first contribution where regularity of conditional measures were obtained in a purely non-smooth framework.
The techniques introduced in \cite{biacava:streconv} permitted also to threat such regularity issues in the infinite dimensional setting of Wiener space; see \cite{cava:Wiener}.
\medskip

This short introduction should suggest that $L^{1}$-Optimal Transportation permits to obtain an efficient dimensional reduction 
together with a localization of the ``smoothness'' of the space for very general metric measure spaces.
We now make a short introduction also to the L\'evy-Gromov isoperimetric inequality.

\medskip
\subsection{L\'evy-Gromov isoperimetric inequality}

The L\'evy-Gromov isoperimetric inequality \cite[Appendix C]{Gro} can be stated as follows: if $E$ is a (sufficiently regular) 
subset of a Riemannian manifold $(M^N,g)$ with dimension $N$ and Ricci bounded below  by $K>0$, then
\begin{equation}\label{eq:LevyGromov}
\frac{|\partial E|}{|M|}\geq \frac{|\partial B|}{|S|},
\end{equation}
where $B$ is a spherical cap in the model sphere $S$, i.e. the $N$-dimensional sphere with constant Ricci curvature equal to $K$,  and $|M|,|S|,|\partial E|, |\partial B|$  denote the appropriate $N$ or $N-1$ dimensional volume, and where $B$ is chosen so that
$|E|/|M|=|B|/|S|$. As $K >0$ both $M$ and $S$ are compact and their volume is finite; hence the previous equality and \eqref{eq:LevyGromov} makes sense.
In other words, the L\'evy-Gromov isoperimetric inequality states that  isoperimetry in $(M,g)$ is at least as strong as in the model space $S$.
\medskip

A general introduction on the isoperimetric problem goes beyond the scopes of this note; here it is worth mentioning 
that a complete description of isoperimetric inequality in spaces admitting singularities is quite an hard task and the bibliography reduces to \cite{MilRot,  MR, MorPol}.
See also \cite[Appendix H]{EiMe} for more details.
We also include the following reference to the isoperimetric problem corresponding to different approaches: for a geometric measure theory approach see \cite{Mor}; 
for the point of view of optimal transport see \cite{FiMP, Vil}; for  the connections with convex and integral geometry see  \cite{BurZal}; 
for the recent quantitative forms see \cite{CL, FuMP} and finally for an overview of the more geometric aspects see \cite{Oss, Rit, Ros}.
\medskip

Coming back to L\'evy-Gromov isoperimetric inequality, it makes sense naturally also in the broader class of metric measure spaces, i.e. triples $(X,\sfd,\mm)$ where  
$(X,\sfd)$ is complete and separable and $\mm$ is a Radon measure over $X$.
Indeed the volume of a Borel set is replaced by its  $\mm$-measure, $\mm(E)$; 
the boundary area of the smooth framework instead can be replaced by the Minkowski content:
\begin{equation}\label{def:MinkCont}
\mm^+(E):=\liminf_{\ve\downarrow 0} \frac{\mm(E^\ve)- \mm(E)}{\ve},
\end{equation}
where $E^{\ve}:=\{x \in X \,:\, \exists y \in  E \, \text{ such that } \, \sfd(x,y)< \ve \}$ is the $\ve$-neighborhood of $E$ with respect to the metric $\sfd$;
the natural analogue of ``dimension $N$ and Ricci bounded below  by $K>0$'' is encoded in the so-called Riemannian Curvature Dimension condition, 
$\RCD^{*}(K,N)$ for short. As normalization factors appears in \eqref{eq:LevyGromov}, it is also more convenient 
to directly consider the case $\mm(X) = 1$.

\medskip

So the L\'evy-Gromov isoperimetric problem for a m.m.s. $(X,\sfd,\mm)$  with $\mm(X) = 1$ can be formulated as follows: \\

\noindent
\emph{ 
Find the  largest function $\cI_{K,N}:[0,1]\to \R^+$ such that for every Borel subset $E\subset X$ it holds 
$$
\mm^{+}(E) \geq \cI_{K,N}(\mm(E)), 
$$
with $\cI_{K,N}$ depending on $N, K \in \R$ with $K>0$ and $N>1$.
}

\medskip
Then in \cite{CM1} (Theorem 1.2) the author with A. Mondino proved the non-smooth L\'evy-Gromov isoperimetric inequality \eqref{eq:LevyGromov}

\begin{theorem}[L\'evy-Gromov in $\RCD^*(K,N)$-spaces, Theorem 1.2 of \cite{CM1}] \label{thm:LG}
Let  $(X,\sfd,\mm)$ be an $\RCD^*(K,N)$ space for some $N\in \N$ and $K>0$ and $\mm(X)=1$. Then for every Borel subset $E\subset X$ it holds
$$
\mm^+(E)\geq  \frac{|\partial B|}{|S|},
$$
where $B$ is a spherical cap in the model sphere $S$ (the $N$-dimensional sphere with constant Ricci curvature equal to $K$) chosen so that $|B|/|S|=\mm(E)$.
\end{theorem}

We refer to Theorem 1.2 of \cite{CM1} (or Theorem 6.6) for the more general statement.

\medskip 

The link between Theorem \ref{thm:LG} and the first part of the Introduction, where the Monge problem was discussed, stands in the techniques used to prove Theorem \ref{thm:LG}.

The main obstacle to L\'evy-Gromov type inequalities in the non-smooth metric measure spaces setting is that the 
previously known proofs rely on regularity properties of isoperimetric regions and on powerful results of geometric measure theory (see for instance \cite{Gro,Mor}) 
that are out of disposal in the framework of metric measure spaces.
The recent paper of B. Klartag \cite{klartag} permitted to obtain a proof of the L\'evy-Gromov isoperimetric inequality, still  in the framework of smooth Riemannian manifolds, 
avoiding regularity of optimal shapes and using instead an optimal transportation argument involving $L^1$-Optimal Transportation and ideas of convex geometry. 
This approach goes back to Payne-Weinberger \cite{PW} and was later developed by Gromov-Milman \cite{GrMi}, 
Lov\'asz-Simonovits \cite{LoSi} and Kannan-Lov\'asz-Simonovits \cite{KaLoSi}; 
it consists in reducing a multi-dimensional problem, to easier one-dimensional problems. B. Klartag's contribution was to observe that 
a suitable $L^{1}$-Optimal Transportation problem produces what he calls a \emph{needle decomposition} (in our terminology will be called disintegration) 
that localize (or reduce) the proof of the isoperimetric inequality to the proof of a family of one-dimensional isoperimetric inequalities; 
also the regularity of the space is localized.

The approach of \cite{klartag} does not rely on the regularity of the isoperimetric region, nevertheless it still heavily  
makes use of the smoothness of the ambient space to obtain the localization; in particular
it makes use of sharp properties of the geodesics in terms of Jacobi fields and estimates on the second fundamental forms of suitable level sets, 
all objects that are still not enough understood in general metric measure space in order to repeat the same arguments.

Hence to apply the localization technique to the L\'evy-Gromov isoperimetric inequality in singular spaces,
structural properties of geodesics and of $L^1$-optimal transportation 
have to be understood also in the general framework of metric measure spaces.
Such a program already started in the previous work of the author with S. Bianchini \cite{biacava:streconv} and of the 
author \cite{cava:MongeRCD, cava:decomposition}. 
Finally with A. Mondino in \cite{CM1} we obtained the general result permitting to obtained the L\'evy-Gromov isoperimetric inequality.

\medskip
\subsection{Outline}
The outline of the paper goes as follows:
Section \ref{S:preliminaries} contains all the basic material on Optimal Transportation and the theory of Lott-Sturm-Villani spaces, that is 
metric measure spaces verifying the Curvature Dimension condition, $\CD(K,N)$ for short. It also covers some basics on isoperimetric inequality, 
Disintegration Theorem and selection theorems we will use during the paper.
In Section \ref{S:transportset} we prove all the structure results on the building block 
of $L^{1}$-Optimal Transportation, the $\sfd$-cyclically monotone sets. Here no curvature assumption enters. 
In Section \ref{S:cyclically} we show that the aforementioned sets induce a partition of almost all transport, provided the space
enjoies a stronger form of the essentially non-branching condition; we also show that each element of the partition is a geodesic (and therefore a one-dimensional set).
Section \ref{S:ConditionalMeasures} contains all the regularity results of conditional 
measures of the disintegration induced by the $L^{1}$-Optimal Transportation problem. In particular we will present three assumptions, each one implying the previous one, 
yielding three increasing level of regularity of the conditional measures. Finally in Section \ref{S:application} we collect the 
consequences of the regularity results of Section \ref{S:ConditionalMeasures}; in particular we first show the existence of a 
solution of the Monge problem under very general regularity assumption (Theorem \ref{T:mongeff}) and finally we go back to the L\'evy-Gromov 
isoperimetric inequality (Theorem \ref{T:iso}). 

\bigskip


\section{Preliminaries}\label{S:preliminaries}

In what follows we say that a triple $(X,\sfd, \mm)$ is a metric measure space, m.m.s. for short, 
if $(X, \sfd)$ is a complete and separable metric space and $\mm$ is positive Radon measure over $X$. 
For this paper we will only be concerned with m.m.s. with $\mm$ probability measure, that is $\mm(X) =1$.
The space of all Borel probability measures over $X$ will be denoted by $\mathcal{P}(X)$.

A metric space is a geodesic space if and only if for each $x,y \in X$ 
there exists $\gamma \in \Geo(X)$ so that $\gamma_{0} =x, \gamma_{1} = y$, with
$$
\Geo(X) : = \{ \gamma \in C([0,1], X):  \sfd(\gamma_{s},\gamma_{t}) = |s-t| \sfd(\gamma_{0},\gamma_{1}), \text{ for every } s,t \in [0,1] \}.
$$
It follows from the metric version of the Hopf-Rinow Theorem (see Theorem 2.5.28 of \cite{BBI}) that for complete geodesic spaces,
local completeness is equivalent to properness (a metric space is proper if every closed ball is compact).

So we assume the ambient space $(X,\sfd)$ to be proper and geodesic, hence also complete and separable. 
Moreover we assume $\mm$ to be a proability measure, i.e. $\mm(X)=1$.

\medskip

We denote by $\mathcal{P}_{2}(X)$ the space of probability measures with finite second moment  endowed with the $L^{2}$-Wasserstein distance  $W_{2}$ defined as follows:  for $\mu_0,\mu_1 \in \mathcal{P}_{2}(X)$ we set
\begin{equation}\label{eq:Wdef}
  W_2^2(\mu_0,\mu_1) = \inf_{ \pi} \int_{X\times X} \sfd^2(x,y) \, \pi(dxdy),
\end{equation}
where the infimum is taken over all $\pi \in \mathcal{P}(X \times X)$ with $\mu_0$ and $\mu_1$ as the first and the second marginal, called the set of transference plans.
The set of transference plans realizing the minimum in \eqref{eq:Wdef} will be called the set of optimal transference plans.
Assuming the space $(X,\sfd)$ to be geodesic, also the space $(\mathcal{P}_2(X), W_2)$ is geodesic. 

Any geodesic $(\mu_t)_{t \in [0,1]}$ in $(\mathcal{P}_2(X), W_2)$  can be lifted to a measure $\nu \in {\mathcal {P}}(\Geo(X))$, 
so that $({\rm e}_t)_\sharp \, \nu = \mu_t$ for all $t \in [0,1]$. 
Here for any $t\in [0,1]$,  ${\rm e}_{t}$ denotes the evaluation map: 
$$
  {\rm e}_{t} : \Geo(X) \to X, \qquad {\rm e}_{t}(\gamma) : = \gamma_{t}.
$$

Given $\mu_{0},\mu_{1} \in \mathcal{P}_{2}(X)$, we denote by 
$\Opt(\mu_{0},\mu_{1})$ the space of all $\nu \in \mathcal{P}(\Geo(X))$ for which $({\rm e}_0,{\rm e}_1)_\sharp\, \nu$ 
realizes the minimum in \eqref{eq:Wdef}. If $(X,\sfd)$ is geodesic, then the set  $\Opt(\mu_{0},\mu_{1})$ is non-empty for any $\mu_0,\mu_1\in \mathcal{P}_2(X)$.
It is worth also introducing the subspace of $\mathcal{P}_{2}(X)$
formed by all those measures absolutely continuous with respect with $\mm$: it is denoted by $\mathcal{P}_{2}(X,\sfd,\mm)$.


\subsection{Geometry of metric measure spaces}\label{Ss:geom}
Here we briefly recall the synthetic notions of lower Ricci curvature bounds, for more detail we refer to  \cite{BS10,lottvillani:metric,sturm:I, sturm:II, Vil}.

In order to formulate the curvature properties for $(X,\sfd,\mm)$ we introduce the following distortion coefficients: given two numbers $K,N\in \erre$ with $N\geq0$, we set for $(t,\theta) \in[0,1] \times \erre_{+}$, 
\begin{equation}\label{E:sigma}
\sigma_{K,N}^{(t)}(\theta):= 
\begin{cases}
\infty, & \textrm{if}\ K\theta^{2} \geq N\pi^{2}, \crcr
\displaystyle  \frac{\sin(t\theta\sqrt{K/N})}{\sin(\theta\sqrt{K/N})} & \textrm{if}\ 0< K\theta^{2} <  N\pi^{2}, \crcr
t & \textrm{if}\ K \theta^{2}<0 \ \textrm{and}\ N=0, \ \textrm{or  if}\ K \theta^{2}=0,  \crcr
\displaystyle   \frac{\sinh(t\theta\sqrt{-K/N})}{\sinh(\theta\sqrt{-K/N})} & \textrm{if}\ K\theta^{2} \leq 0 \ \textrm{and}\ N>0.
\end{cases}
\end{equation}

We also set, for $N\geq 1, K \in \R$ and $(t,\theta) \in[0,1] \times \erre_{+}$
\begin{equation} \label{E:tau}
\tau_{K,N}^{(t)}(\theta): = t^{1/N} \sigma_{K,N-1}^{(t)}(\theta)^{(N-1)/N}.
\end{equation}

%
%
%
%
%
%
%
%
%

As we will consider only the case of essentially non-branching spaces, we recall the following definition. 
\begin{definition}\label{D:essnonbranch}
A metric measure space $(X,\sfd, \mm)$ is \emph{essentially non-branching} if and only if for any $\mu_{0},\mu_{1} \in \mathcal{P}_{2}(X)$,
with $\mu_{0}$ absolutely continuous with respect to $\mm$, any element of $\Opt(\mu_{0},\mu_{1})$ is concentrated on a set of non-branching geodesics.
\end{definition}

A set $F \subset \Geo(X)$ is a set of non-branching geodesics if and only if for any $\gamma^{1},\gamma^{2} \in F$, it holds:
$$
\exists \;  \bar t\in (0,1) \text{ such that } \ \forall t \in [0, \bar t\,] \quad  \gamma_{ t}^{1} = \gamma_{t}^{2}   
\quad 
\Longrightarrow 
\quad 
\gamma^{1}_{s} = \gamma^{2}_{s}, \quad \forall s \in [0,1].
$$

\begin{definition}[$\CD$ condition]\label{D:CD}
An essentially non-branching m.m.s. $(X,\sfd,\mm)$ verifies $\mathsf{CD}(K,N)$  if and only if for each pair 
$\mu_{0}, \mu_{1} \in \mathcal{P}_{2}(X,\sfd,\mm)$ there exists $\nu \in \Opt(\mu_{0},\mu_{1})$ such that
\begin{equation}\label{E:CD}
\r_{t}^{-1/N} (\gamma_{t}) \geq  \tau_{K,N}^{(1-t)}(\sfd( \gamma_{0}, \gamma_{1}))\r_{0}^{-1/N}(\gamma_{0}) 
 + \tau_{K,N}^{(t)}(\sfd(\gamma_{0},\gamma_{1}))\r_{1}^{-1/N}(\gamma_{1}), \qquad \nu\text{-a.e.} \, \gamma \in \Geo(X),
\end{equation}
for all $t \in [0,1]$, where $({\rm e}_{t})_\sharp \, \nu = \r_{t} \mm$.
\end{definition}

For the general definition of $\CD(K,N)$ see \cite{lottvillani:metric, sturm:I, sturm:II}.

\begin{remark}\label{R:CDN-1}

It is worth recalling that if $(M,g)$ is a Riemannian manifold of dimension $n$ and 
$h \in C^{2}(M)$ with $h > 0$, then the m.m.s. $(M,g,h \, vol)$ verifies $\CD(K,N)$ with $N\geq n$ if and only if  (see Theorem 1.7 of \cite{sturm:II})
$$
Ric_{g,h,N} \geq  K g, \qquad Ric_{g,h,N} : =  Ric_{g} - (N-n) \frac{\nabla_{g}^{2} h^{\frac{1}{N-n}}}{h^{\frac{1}{N-n}}}.  
$$
In particular if $N = n$ the generalized Ricci tensor $Ric_{g,h,N}= Ric_{g}$ makes sense only if $h$ is constant. 

Another important case is when $I \subset \R$ is any interval, $h \in C^{2}(I)$ 
and $\mathcal{L}^{1}$ is the one-dimensional Lebesgue measure; then the m.m.s. $(I ,|\cdot|, h \mathcal{L}^{1})$ verifies $\CD(K,N)$ if and only if  
\begin{equation}\label{E:CD-N-1}
\left(h^{\frac{1}{N-1}}\right)'' + \frac{K}{N-1}h^{\frac{1}{N-1}} \leq 0,
\end{equation}
and verifies $\CD(K,1)$ if and only if $h$ is constant. Inequality \eqref{E:CD-N-1} has also a non-smooth counterpart; if we drop the smoothness assumption on $h$ 
it can be proven that the m.m.s. $(I ,|\cdot|, h \mathcal{L}^{1})$ verifies $\CD(K,N)$ if and only if   
\begin{equation}\label{E:curvdensmmR}
h( (1-s)  t_{0}  + s t_{1} )^{1/(N-1)}  
 \geq \sigma^{(1-s)}_{K,N-1}(t_{1} - t_{0}) h (t_{0})^{1/(N-1)} + \sigma^{(s)}_{K,N-1}(t_{1} - t_{0}) h (t_{1})^{1/(N-1)},
\end{equation}
that is the formulation in the sense of distributions of the  differential inequality
$$
\left(h^{\frac{1}{N-1}}\right)'' + \frac{K}{N-1}h^{\frac{1}{N-1}} \leq 0.
$$
Recall indeed that $s \mapsto \sigma^{(s)}_{K,N-1}(\theta)$ solves in the classical sense $f'' + (t_{1}-t_{0})^{2} \frac{K}{N-1}f = 0$.
\end{remark}

%

We also mention the more recent Riemannian curvature dimension condition $\RCD^{*}(K,N)$. In the infinite dimensional case, i.e. $N = \infty$, it was introduced \cite{AGS11b}. 
The class $\RCD^{*}(K,N)$ with $N<\infty$ has been proposed in \cite{G15} and deeply investigated 
in \cite{AGS, EKS} and \cite{AMS}. We refer to these papers and references therein for a general account 
on the synthetic formulation of Ricci curvature lower bounds for metric measure spaces. 

Here we only mention that $\RCD^{*}(K,N)$ condition 
is an enforcement of the so called reduced curvature dimension condition, denoted by $\CD^{*}(K,N)$, that has been introduced in \cite{BS10}: 
in particular the additional condition is that the Sobolev space $W^{1,2}(X,\mm)$ is an Hilbert space, see \cite{G15, AGS11a, AGS11b}.

The reduced $\CD^{*}(K,N)$ condition asks for the same inequality \eqref{E:CD} of $\CD(K,N)$ but  the
coefficients $\tau_{K,N}^{(t)}(\sfd(\gamma_{0},\gamma_{1}))$ and $\tau_{K,N}^{(1-t)}(\sfd(\gamma_{0},\gamma_{1}))$ 
are replaced by $\sigma_{K,N}^{(t)}(\sfd(\gamma_{0},\gamma_{1}))$ and $\sigma_{K,N}^{(1-t)}(\sfd(\gamma_{0},\gamma_{1}))$, respectively.

Hence while the distortion coefficients of the $\CD(K,N)$ condition 
are formally obtained imposing one direction with linear distortion and $N-1$ directions affected by curvature, 
the $\CD^{*}(K,N)$ condition imposes the same volume distortion in all the $N$ directions.

For both definitions there is a local version that is of some relevance for our analysis. Here we state only the local formulation $\mathsf{CD}(K,N)$, 
being clear what would be the one for $\mathsf{CD}^{*}(K,N)$.

\begin{definition}[$\CD_{loc}$ condition]\label{D:loc}
An essentially non-branching m.m.s. $(X,\sfd,\mm)$ satisfies $\CD_{loc}(K,N)$ if for any point $x \in X$ there exists a neighborhood $X(x)$ of $x$ such that for each pair 
$\mu_{0}, \mu_{1} \in \mathcal{P}_{2}(X,\sfd,\mm)$ supported in $X(x)$
there exists $\nu \in \Opt(\mu_{0},\mu_{1})$ such that \eqref{E:CD} holds true for all $t \in [0,1]$.
The support of $({\rm e}_{t})_\sharp \, \nu$ is not necessarily contained in the neighborhood $X(x)$.
\end{definition}

One of the main properties of the reduced curvature dimension condition is the globalization one:  
under the essentially non-branching property,  $\mathsf{CD}^{*}_{loc}(K,N)$ and $\mathsf{CD}^{*}(K,N)$ are equivalent (see \cite[Corollary 5.4]{BS10}), i.e. the 
$\mathsf{CD}^{*}$-condition verifies the local-to-global property.

We also recall a few relations between $\CD$ and $\CD^{*}$.
It is known by \cite[Theorem 2.7]{GigliMap} that, if $(X,\sfd,\mm)$ is a non-branching metric measure space 
verifying $\CD(K,N)$ and $\mu_{0}, \mu_{1} \in \mathcal{P}(X)$ with $\mu_{0}$ absolutely continuous with respect to $\mm$, 
then there exists a unique optimal map $T : X \to X$ such $(id, T)_\sharp\, \mu_{0}$ realizes the minimum in \eqref{eq:Wdef} and the set 
$\Opt(\mu_{0},\mu_{1})$ contains only one element. The same proof holds if one replaces the non-branching assumption with the more general 
one of essentially non-branching, see for instance \cite{GRS2013}.

\bigskip

\subsection{Isoperimetric profile function}

Given a m.m.s. $(X,\sfd,\mm)$ as above and  a Borel subset $A\subset X$, let $A^{\ve}$ denote the $\ve$-tubular neighborhood 
$$
A^{\ve}:=\{x \in X \,:\, \exists y \in A \text{ such that } \sfd(x,y) < \ve \}. 
$$
The Minkowski (exterior) boundary measure $\mm^+(A)$  is defined by
\begin{equation}\label{eq:MinkCont}
\mm^+(A):=\liminf_{\ve\downarrow 0} \frac{\mm(A^\ve)-\mm(A)}{\ve}.
\end{equation}
The \emph{isoperimetric profile}, denoted by  ${\cI}_{(X,\sfd,\mm)}$, is defined as the point-wise maximal function so that $\mm^+(A)\geq \cI_{(X,\sfd,\mm)}(\mm(A))$  
for every Borel set $A \subset X$, that is
\begin{equation}\label{E:profile}
\cI_{(X,\sfd,\mm)}(v) : = \inf \big\{ \mm^{+}(A) \colon A \subset X \, \textrm{ Borel}, \, \mm(A) = v   \big\}.
\end{equation}

If $K>0$ and $N\in \N$, by the L\'evy-Gromov isoperimetric inequality \eqref{eq:LevyGromov} we know that, for $N$-dimensional smooth manifolds having Ricci $\geq K$, 
the isoperimetric profile function is bounded below by the one of the $N$-dimensional round sphere of the suitable radius. 
In other words  the \emph{model} isoperimetric profile function is the one of ${\mathbb S}^N$. 
For $N\geq 1, K\in \R$ arbitrary real numbers the situation is more complicated, and just recently E. Milman \cite{Mil} discovered what is the model isoperimetric profile. 
We refer to \cite{Mil} for all the details. Here we just recall the relevance of isoperimetric profile functions for m.m.s. over $(\R, |\cdot|)$:
given $K\in \R, N\in[1,+\infty)$ and $D\in (0,+\infty]$, consider the function
\begin{equation}\label{defcI}
\mathcal{I}_{K,N,D}(v) : = \inf \left\{ \mu^{+}(A) \colon A\subset \R, \,\mu(A) = v, \, \mu \in \mathcal{F}_{K,N,D}  \right\},
\end{equation}
where  $\mathcal{F}_{K,N,D}$ denotes the set of $\mu \in \mathcal{P}(\R)$ such that  $\supp(\mu) \subset [0,D]$   
and $\mu = h \cdot \mathcal{L}^{1}$ with $h \in C^{2}((0,D))$ satisfying
\begin{equation}\label{eq:DiffIne}
\left( h^{\frac{1}{N-1}} \right)'' + \frac{K}{N-1} h^{\frac{1}{N-1}} \leq 0 \quad \text{if }N \in (1,\infty), \quad h\equiv \textrm{const} \quad \text{if }N=1.
\end{equation}
Then from \cite[Theorem 1.2, Corollary 3.2]{Mil} it follows that for $N$-dimensional smooth manifolds having Ricci $\geq K$, with $K\in \R$ 
arbitrary real number, and diameter $D$,  the isoperimetric profile function is bounded below by $\mathcal{I}_{K,N,D}$ and the bound is sharp.
This also justifies the notation.

\medskip

Going back to non-smooth metric measure spaces (what follows is taken from \cite{CM1}), it is necessary to consider the following broader family of measures:
\begin{eqnarray}
\mathcal{F}^{s}_{K,N,D} : = \{ \mu \in \mathcal{P}(\R) : &\supp(\mu) \subset [0,D], \, \mu = h_{\mu} \mathcal{L}^{1},\,
h_{\mu}\, \textrm{verifies} \, \eqref{E:curvdensmmR} \ \textrm{and is continuous if } N\in (1,\infty), \nonumber \\ 
& \quad h_{\mu}\equiv \textrm{const} \text{ if }N=1   \},
\end{eqnarray}
and  the corresponding  comparison \emph{synthetic} isoperimetric profile:  
$$
\mathcal{I}^{s}_{K,N,D}(v) : = \inf \left\{ \mu^{+}(A) \colon A\subset \R, \,\mu(A) = v, \, \mu \in \mathcal{F}^{s}_{K,N,D}  \right\},
$$
where $\mu^{+}(A)$ denotes the Minkowski content defined  in  \eqref{eq:MinkCont}.
The term synthetic refers to $\mu \in \mathcal{F}^{s}_{K,N,D}$ meaning that the Ricci curvature bound is satisfied in its synthetic formulation:
if $\mu = h \cdot \mathcal{L}^{1}$, then $h$ verifies \eqref{E:curvdensmmR}.

\medskip

We have already seen that $\mathcal{F}_{K,N,D} \subset \mathcal{F}^{s}_{K,N,D}$; actually one can prove  that 
$\mathcal{I}^{s}_{K,N,D}$ coincides with  its smooth counterpart $\mathcal{I}_{K,N,D}$ for every volume $v \in [0,1]$ via a smoothing argument. 
We therefore need the following approximation result. In order to state it let us recall that a standard mollifier in $\R$ is a non negative $C^\infty(\R)$ 
function $\psi$ with compact support in $[0,1]$ such  that $\int_{\R} \psi = 1$. 
\begin{lemma}[Lemma 6.2, \cite{CM1}]\label{lem:approxh}
Let  $D \in (0,\infty)$ and let  $h:[0,D] \to [0,\infty)$ be a continuous function. Fix $N\in (1,\infty)$ and for $\ve>0$ define
\begin{equation}
h_{\ve}(t):=[h^{\frac{1}{N-1}}\ast \psi_{\ve} (t)]^{N-1}  := \left[ \int_{\R} h(t-s)^{\frac{1}{N-1}}  \psi_{\ve} (s) \, d s\right]^{N-1} 
										=  \left[ \int_{\R} h(s)^{\frac{1}{N-1}}  \psi_{\ve} (t-s) \, d s\right]^{N-1},
\end{equation}
where $\psi_\ve(x)=\frac{1}{\ve} \psi(x/\ve)$ and $\psi$ is a standard mollifier function. The following properties hold:
\begin{enumerate}
	\item $h_{\ve}$ is a non-negative $C^\infty$ function with support in $[-\ve, D+\ve]$; \medskip
	\item $h_{\ve}\to h$ uniformly  as $\ve \downarrow 0$, in particular $h_{\ve} \to h$ in $L^{1}$. \medskip
	\item If $h$ satisfies the convexity condition \eqref{E:curvdensmm} corresponding to the above fixed $N>1$ 
		and some $K \in \R$ then also $h_{\ve}$ does. In particular $h_{\ve}$ satisfies the differential inequality \eqref{eq:DiffIne}.
\end{enumerate}
\end{lemma}

Using this approximation one can prove the following 
\medskip

\begin{theorem}[Theorem 6.3, \cite{CM1}]\label{thm:I=Is}
For every $v\in [0,1]$, $K \in \R$, $N\in [1,\infty)$, $D\in (0,\infty]$ it holds $\mathcal{I}^{s}_{K,N,D}(v)=\mathcal{I}_{K,N,D}(v)$.
\end{theorem}

\medskip

\subsection{Disintegration of measures}
We include here a version of Disintegration Theorem that we will use. 
We will follow Appendix A of \cite{biacara:extreme} where a  self-contained approach (and a proof) of Disintegration Theorem in countably generated measure spaces can be found. 
An even more general version of Disintegration Theorem can be found in Section 452 of \cite{Fre:measuretheory4}.

Recall that a $\sigma$-algebra is \emph{countably generated} if there exists a countable family of sets so that 
the $\sigma$-algebra  coincide with the smallest $\sigma$-algebra containing them.

\medskip

Given a measurable space $(X, \mathscr{X})$, i.e.  $\mathscr{X}$ is  a $\sigma$-algebra of subsets of $X$, 
and a function $\QQ : X \to Q$, with $Q$ general set, we can endow $Q$ with the \emph{push forward $\sigma$-algebra} $\mathscr{Q}$ of $\mathscr{X}$:
$$
C \in \mathscr{Q} \quad \Longleftrightarrow \quad \QQ^{-1}(C) \in \mathscr{X},
$$
which could be also defined as the biggest $\sigma$-algebra on $Q$ such that $\QQ$ is measurable. 
Moreover given a probability measure  $\mm$  on $(X,\mathscr{X})$, define a probability  
measure $\qq$ on $(Q,\mathscr{Q})$  by push forward via $\QQ$, i.e. $\qq := \QQ_\sharp \, \mm$.

This general scheme fits with the following situation: given a measure space $(X,\mathscr{X},\mm)$, suppose 
a partition of $X$ is given in the form $\{ X_{q}\}_{q\in Q}$, $Q$ is the set of indices and $\QQ : X \to Q$ is the quotient map, i.e.  
$$
q = \QQ(x) \iff x \in X_{q}.
$$
Following the previous scheme, we can consider also the quotient $\sigma$-algebra $\mathscr{Q}$ and the quotient measure $\qq$ obtaining 
the quotient measure space $(Q, \mathscr{Q}, \qq)$. 

\medskip

\begin{definition}
\label{defi:dis}
A \emph{disintegration} of $\mm$ \emph{consistent with} $\QQ$ is a map 
$$
Q \ni q \longmapsto \mm_{q} \in \mathcal{P}(X,\mathscr{X})
$$
such that the following hold:
\begin{enumerate}
\item  for all $B \in \mathscr{X}$, the map $\mm_{\cdot}(B)$ is $\qq$-measurable;
\item for all $B \in \mathscr{X}, C \in \mathscr{Q}$ satisfies the consistency condition
$$
\mm \left(B \cap \QQ^{-1}(C) \right) = \int_{C} \mm_{q}(B)\, \qq(dq).
$$
\end{enumerate}
A disintegration is \emph{strongly consistent with respect to $\QQ$} if for $\qq$-a.e. $q \in Q$ we have $\mm_{q}(\QQ^{-1}(q))=1$.
The measures $\mm_q$ are called \emph{conditional probabilities}.
\end{definition}

When the map $\QQ$ is induced by a partition of $X$ as before, we will directly say that the disintegration is consistent with the partition, meaning that 
the disintegration is consistent with the quotient map $\QQ$ associated to the partition $X = \cup_{q\in Q} X_{q}$.

We now report Disintegration Theorem.

\medskip
\begin{theorem}[Theorem A.7, Proposition A.9 of \cite{biacara:extreme}]\label{T:disintegrationgeneral}
\label{T:disintr}
Assume that $(X,\mathscr{X},\rho)$ is a countably generated probability space and  $X = \cup_{q \in Q}X_{q}$ is a partition of $X$. 
\medskip

Then the quotient probability space $(Q, \mathscr{Q},\qq)$ is essentially countably generated and 
there exists a unique disintegration $q \mapsto \mm_{q}$ consistent with the partition $X = \cup_{q\in Q} X_{q}$.

\medskip

The disintegration is strongly consistent if and only if there exists a $\mm$-section $S \in \mathscr{X}$ such that the $\sigma$-algebra $\mathscr{S}$ contains $\mathcal{B}(S)$. 
\end{theorem}

We expand the statement of Theorem \ref{T:disintegrationgeneral}.\\
In the measure space $(Q, \mathscr{Q},\qq)$, the $\sigma$-algebra $\mathscr{Q}$ is \emph{essentially countably generated} if, by definition, 
there exists a countable family of sets $Q_{n} \subset Q$ such that for any $C \in \mathscr{Q}$ there exists $\hat C \in \hat{\mathscr{Q}}$, 
where $\hat{\mathscr{Q}}$ is the $\sigma$-algebra generated by $\{ Q_{n} \}_{n \in \N}$, such that $\qq(C\, \Delta \, \hat C) = 0$.

Uniqueness is understood in the following sense: if $q\mapsto \mm^{1}_{q}$ and $q\mapsto \mm^{2}_{q}$ are two consistent disintegrations then 
$\mm^{1}_{q}=\mm^{2}_{q}$ for $\qq$-a.e. $q \in Q$.

Finally, a set $S$ is a section for the partition $X = \cup_{q}X_{q}$ if for any $q \in Q$ there exists a unique $x_{q} \in S \cap X_{q}$.
A set $S_{\mm}$ is an $\mm$-section if there exists $Y \subset X$ with $\mm(X \setminus Y) = 0$ such that the partition $Y = \cup_{q} (X_{q} \cap Y)$
has section $S_{\mm}$. Once a section (or an $\mm$-section) is given, one can obtain the measurable space $(S,\mathscr{S})$ by pushing forward the $\sigma$-algebra 
$\mathscr{X}$ on $S$ via the map that associates to any $X_{q} \ni x \mapsto x_{q} = S \cap X_{q}$.

\medskip

\bigskip
\bigskip

\section{Transport set}\label{S:transportset}
The following setting is fixed once for all: 
\medskip

\begin{center}
$(X,\sfd,\mm)$ is a fixed metric measure space with $\mm(X)=1$ such that \\
the ambient metric space $(X, \sfd)$ is geodesic and proper (hence complete and separable).
\end{center}

\medskip
Let $\f : X \to \R$ be any $1$-Lipschitz function. Here  we present some useful results (all of them already presented in \cite{biacava:streconv}) 
concerning the $\sfd$-cyclically monotone set associated with $\f$:
\begin{equation}\label{E:Gamma}
\Gamma : = \{ (x,y) \in X\times X : \f(x) - \f(y) = \sfd(x,y) \},
\end{equation}
that can be seen as the set of couples moved by $\f$ with maximal slope.
Recall that a set $\Lambda \subset X \times X$ is said to be $\sfd$-cyclically monotone if for any finite set of points $(x_{1},y_{1}),\dots,(x_{N},y_{N})$ it holds
$$
\sum_{i = 1}^{N} \sfd(x_{i},y_{i}) \leq \sum_{i = 1}^{N} \sfd(x_{i},y_{i+1}),
$$
with the convention that $y_{N+1} = y_{1}$. \medskip

The following lemma is a consequence of the $\sfd$-cyclically monotone structure of $\Gamma$.


\begin{lemma}\label{L:cicli}
Let $(x,y) \in X\times X$ be an element of $\Gamma$. Let $\gamma \in \Geo(X)$ be such that $\gamma_{0} = x$ 
and $\gamma_{1}=y$. Then
$$
(\gamma_{s},\gamma_{t}) \in \Gamma, 
$$
for all $0\leq s \leq t \leq 1$.
\end{lemma}
\begin{proof}
Take $0\leq s \leq t \leq 1$ and note that
\begin{align*}
 \f(\gamma_{s})& - \f(\gamma_{t})\crcr
 = &~ \f(\gamma_{s}) - \f(\gamma_{t}) + \f(\gamma_{0}) - \f(\gamma_{0}) + \f(\gamma_{1}) - \f(\gamma_{1})\crcr
\geq &~ \sfd(\gamma_{0},\gamma_{1}) - \sfd(\gamma_{0},\gamma_{s}) - \sfd(\gamma_{t},\gamma_{1}) \crcr
= &~ \sfd(\gamma_{s},\gamma_{t}).
\end{align*}
The claim follows.
\end{proof}

It is natural then to consider the set of geodesics $G \subset \Geo(X)$ such that 
$$
\gamma \in G \iff \{ (\gamma_{s},\gamma_{t}) : 0\leq s \leq t \leq 1  \} \subset \Gamma,
$$
that is $G : = \{ \gamma \in \Geo(X) : (\gamma_{0},\gamma_{1}) \in \Gamma \}$.  
We now recall some basic definitions of the $L^{1}$-optimal transportation theory that will be needed to describe the structure of $\Gamma$.

\begin{definition}\label{D:transport}
We define the set of \emph{transport rays} by 
$$
R := \Gamma \cup \Gamma^{-1},
$$
where $\Gamma^{-1}:= \{ (x,y) \in X \times X : (y,x) \in \Gamma \}$. 
The set of \emph{initial points} and \emph{final points} are defined   respectively by  
\begin{align*}
{\mathfrak a} :=& \{ z \in X: \nexists \, x \in X, (x,z) \in \Gamma, \sfd(x,z) > 0  \}, \crcr
{\mathfrak b} :=& \{ z \in X: \nexists \, x \in X, (z,x) \in \Gamma, \sfd(x,z) > 0 \}.
\end{align*}
The set of \emph{end points} is ${\mathfrak a} \cup {\mathfrak b}$. 
We define the subset of $X$, \emph{transport set with end points}: 
$$
\mathcal{T}_{e} = P_{1}(\Gamma \setminus \{ x = y \}) \cup 
P_{1}(\Gamma^{-1}\setminus \{ x=y \}).
$$
where $\{ x = y\}$ stands for $\{ (x,y) \in X^{2} : \sfd(x,y) = 0 \}$.
\end{definition}

Few comments are in order. Notice that $R$ coincide with $\{(x,y) \in X \times X \colon |\f(x) -\f(y)| = \sfd(x,y) \}$; 
the name transport set with end points for $\T_{e}$ is motivated by the fact that later on we will consider a more regular subset of $\T_{e}$ that will be called transport set;
moreover if $x \in X$ for instance is moved forward but not backward by $\f$, this is translated in $x \in \Gamma$ and $x \notin \Gamma^{-1}$; anyway it belongs to $\T_{e}$. 

We also introduce the following notation that will be used throughout the paper; 
we set $\Gamma(x):=P_2(\Gamma\cap(\{x\}\times X))$ and  $\Gamma^{-1}(x):=P_2(\Gamma^{-1}\cap(\{x\} \times X))$. 
More in general if $F \subset X \times X$, we set  $F(x) = P_2(F \cap (\{x\}\times X))$.

\begin{remark}\label{R:regularity}
Here we discuss the measurability of the sets introduced in Definition \ref{D:transport}.
Since $\f$ is $1$-Lipschitz, $\Gamma$ is closed and therefore $\Gamma^{-1}$ and $R$ are closed as well.
Moreover by assumption the space is proper, hence the sets $\Gamma, \Gamma^{-1}, R$ are $\sigma$-compact (countable union of compact sets).

Then we look at the set of initial and final points:
$$
{\mathfrak a} = P_{2} \left( \Gamma \cap \{ (x,z) \in X\times X : \sfd(x,z) > 0 \} \right)^{c}, \qquad 
{\mathfrak b} = P_{1} \left( \Gamma \cap \{ (x,z) \in X\times X : \sfd(x,z) > 0 \} \right)^{c}.
$$
Since $\{ (x,z) \in X\times X : \sfd(x,z) > 0 \} = \cup_{n} \{ (x,z) \in X\times X : \sfd(x,z) \geq 1/n \}$, it follows that  
it follows that both ${\mathfrak a}$ and ${\mathfrak b}$ are the complement of $\sigma$-compact sets. 
Hence ${\mathfrak a}$ and ${\mathfrak b}$ are Borel sets. Reasoning as before, it follows that $\T_{e}$ is a $\sigma$-compact set.
\end{remark}

\begin{lemma} \label{L:mapoutside}
Let $\pi \in \Pi(\mu_{0},\mu_{1})$ with $\pi(\Gamma) = 1$, then
$$
\pi(\mathcal{T}_e \times \mathcal{T}_e \cup \{x = y\}) = 1.
$$
\end{lemma}

\begin{proof}
It is enough to observe that if $(z,w) \in \Gamma$ with $z \neq w$, then $w \in \Gamma(z)$ and $z \in \Gamma^{-1}(w)$ 
and therefore
$$
(z,w) \in \mathcal{T}_{e}\times \mathcal{T}_{e}.
$$
Hence $\Gamma \setminus \{x = y\}  \subset   \mathcal{T}_e \times \mathcal{T}_e$. Since $\pi(\Gamma) =1$, 
the claim follows.
\end{proof}

As a consequence, $\mu_{0}(\mathcal{T}_e) = \mu_{1}(\mathcal{T}_e)$ and any optimal map $T$ such that 
$T_\sharp \mu_{0} \llcorner_{\mathcal{T}_e}= \mu_{1} \llcorner_{\mathcal{T}_e}$ 
can be extended to an optimal map $T'$ with $ T^{'}_\sharp \mu_{0} = \mu_{1}$ with the same cost by setting
\begin{equation}
\label{E:extere}
T'(x) =
\begin{cases}
T(x), & \textrm{if } x \in \mathcal{T}_e \crcr
x, & \textrm{if } x \notin \mathcal{T}_e.
\end{cases}
\end{equation}

\medskip

It can be proved that the set of transport rays $R$ induces an equivalence relation on a subset of $\mathcal{T}_{e}$. 
It is sufficient to remove from $\mathcal{T}_{e}$
the branching points of geodesics. Then using curvature properties of the space, one can prove that such branching points all have $\mm$-measure zero.

\bigskip


\subsection{Branching structures in the Transport set}
What follows was first presented in \cite{cava:MongeRCD}.
Consider the sets of respectively forward and backward branching points
\begin{align}\label{E:branchingpoints}
	A_{+}: = 	&~\{ x \in \mathcal{T}_{e} : \exists z,w \in \Gamma(x), (z,w) \notin R \}, \nonumber \\ 
	A_{-}	: = 	&~\{ x \in \mathcal{T}_{e} : \exists z,w \in \Gamma(x)^{-1}, (z,w) \notin R \}.
\end{align}
The sets $A_{\pm}$ are $\sigma$-compact sets. Indeed since $(X,\sfd)$ is proper, any open set is $\sigma$-compact.
The main motivation for the definition of $A_{+}$ and $A_{-}$ is contained in the next 

\medskip
\begin{theorem}\label{T:equivalence}
The set of transport rays $R\subset X \times X$ is an equivalence relation on the set
$$
\mathcal{T}_{e} \setminus \left( A_{+} \cup A_{-} \right).
$$
\end{theorem}

\begin{proof}
First, for all $x \in P_{1}(\Gamma)$, $(x,x) \in R$. If $x,y \in \mathcal{T}_{e}$ with $(x,y) \in R$, then by definition of $R$, it follows straightforwardly that $(y,x) \in R$.

So the only property needing a proof is transitivity. Let $x,z,w \in \mathcal{T}_{e} \setminus \left( A_{+} \cup A_{-} \right)$
be such that $(x,z), (z,w) \in R$ with $x,z$ and $w$ distinct points. The claim is $(x,w) \in R$.
So we have 4 different possibilities: the first one is 
\[
z\in \Gamma(x), \quad w \in \Gamma(z).
\]
This immediately implies $w \in \Gamma(x)$ and therefore $(x,w) \in R$.
The second possibility is 
\[
z\in \Gamma(x), \quad z \in \Gamma(w),
\]
that can be rewritten as $(z,x), (z,w) \in \Gamma^{-1}$. Since $z \notin A_{-}$, necessarily $(x,w) \in R$. 
Third possibility: 
\[
x\in \Gamma(z), \quad w \in \Gamma(z),
\]
and since $z \notin A_{+}$ it follows that $(x,w) \in R$.
The last case is 
\[
x\in \Gamma(z), \quad z \in \Gamma(w),
\]
and therefore $x \in \Gamma(w)$, hence $(x,w) \in R$ and the claim follows.
\end{proof}

\bigskip

Next, we show that each equivalence class of $R$ is formed by a single geodesic.

\begin{lemma}\label{L:singlegeo}
For any $x \in \mathcal{T}$ and $z,w \in R(x)$  there exists $\gamma \in G \subset \Geo(X)$ such that 
$$
\{ x, z,w \} \subset \{ \gamma_{s} : s\in [0,1] \}.
$$
If $\hat \gamma \in G$ enjoys the same property, then 
\[
\big( \{ \hat \gamma_{s} : s \in [0,1] \} \cup \{ \gamma_{s} : s \in [0,1] \} \big) \subset  \{ \tilde \gamma_{s} : s \in [0,1] \}
\]
for some $\tilde \gamma \in G$.
\end{lemma}

Since $G = \{ \gamma \in \Geo(X) : (\gamma_{0},\gamma_{1}) \in \Gamma \}$, Lemma \ref{L:singlegeo} states that 
as soon as we fix an element $x$ in $\mathcal{T}_{e} \setminus (A_{+} \cup A_{-})$ and we pick two elements $z,w$ in the same equivalence class of $x$, then these three points are aligned on 
a geodesic $\gamma$ whose image is again all contained in the same equivalence class $R(x)$.

\begin{proof}
Assume that $x, z$ and $w$ are all distinct points otherwise the claim follows trivially. 
We consider different cases.

\medskip
\noindent
\emph{First case:} $z \in \Gamma(x)$ and $w \in \Gamma^{-1}(x)$. \\
By $\sfd$-cyclical monotonicity
$$
\sfd(z,w) \leq \sfd(z,x) + \sfd(x,w) = \f(w) - \f(z) \leq \sfd(z,w).
$$
Hence $z,x$ and $w$ lie on a geodesic. 
\medskip

\noindent
\emph{Second case:} $z,w \in \Gamma(x)$. \\
Without loss of generality $\f(x) \geq \f(w) \geq \f(z)$. Since  in the proof of 
Lemma \ref{L:geoingamma} we have already excluded the case $\f(w) = \f(z)$, we assume $\f(x) > \f(w) > \f(z)$. 
Then if there would not exist any geodesics $\gamma \in G$ with $\gamma_{0} = x$ and $\gamma_{1} = z$ and $\gamma_{s} = w$,
there will be $\gamma \in G$ with $(\gamma_{0},\gamma_{1}) = (x,z)$ and $s \in (0,1)$ such that 
$$
\f(\gamma_{s}) = \f(w), \qquad \gamma_{s} \in \Gamma(x),  \qquad \gamma_{s} \neq w.
$$
As observed in the proof of Lemma \ref{L:geoingamma}, this would imply that  $(\gamma_{s},w) \notin R$ and 
since $x \notin A_{+}$ this would be a contradiction. Hence the second case follows.

The remaining two cases follow with the same reasoning, exchanging the role of $\Gamma(x)$ with the one of $\Gamma^{-1}(x)$.
The second part of the statement follows now easily.
\end{proof}

\bigskip

\bigskip
\bigskip


\section{Cyclically monotone sets}\label{S:cyclically}

Following Theorem \ref{T:equivalence} and Lemma \ref{L:singlegeo}, the next step is to prove that both $A_{+}$ and $A_{-}$ have $\mm$-measure zero, 
that is branching happens on rays with zero $\mm$-measure. 
Already from the statement of this property, it is clear that some regularity assumption on $(X,\sfd,\mm)$ should play a role. We will indeed assume the space to 
enojoy a stronger form of essentially non-branching. 
Recall that the latter is formulated in terms of geodesics of $(\mathcal{P}_{2}(X),W_{2})$ hence of $\sfd^{2}$-cyclically monotone set, while we need 
regularity for the $\sfd$-cyclically monotone set $\Gamma$.
Hence it is necessary to include $\sfd^{2}$-cyclically monotone sets as subset of $\sfd$-cyclically monotone sets.

We present here a strategy introduced by the author in \cite{cava:decomposition, cava:MongeRCD} from where all the material presented in this section is taken.
Section \ref{Ss:structure} contains results from \cite{biacava:streconv} while Section \ref{Ss:balanced} is taken from \cite{CM1}.

\begin{lemma}[Lemma 4.6 of \cite{cava:decomposition}]\label{L:12monotone}
Let $\Delta \subset \Gamma$ be any set so that: 
$$
(x_{0},y_{0}), (x_{1},y_{1}) \in \Delta \quad \Rightarrow \quad (\f(y_{1}) - \f(y_{0}) )\cdot (\f(x_{1}) - \f(x_{0}) ) \geq 0.
$$
Then $\Delta$ is $\sfd^{2}$-cyclically monotone.
\end{lemma}

\begin{proof} 
It follows directly from the hypothesis of the lemma that the set
$$
\Lambda: = \{ (\f(x), \f(y) ) :   (x,y) \in \Delta \} \subset \erre^{2}, 
$$
is monotone in the Euclidean sense. Since $\Lambda \subset \R^{2}$, it is then a standard fact that $\Lambda$ is 
also $|\cdot|^{2}$-cyclically monotone, where $|\cdot|$ denotes the modulus. 
We anyway  include a short proof:
there exists a maximal monotone multivalued function $F$ such that $\Lambda \subset \gr (F)$ and its domain is an interval, say $(a,b)$ with $a$ and $b$ possibly infinite; 
moreover, apart from countably many $x \in \R$, the set $F(x)$ is a singleton. 
Then the following function is well defined:  
$$
\Psi(x) : = \int_{c}^{x} F(s) ds, 
$$
where $c$ is any fixed element of $(a,b)$. Then observe that %
$$
\Psi(z) - \Psi(x) \geq y(z-x), \qquad \forall \ z,x \in (a,b),
$$
where $y$ is any element of $F(x)$. In particular this implies that $\Psi$ is convex and $F(x)$ is a subset of its sub-differential. 
In particular $\Lambda$ is $|\cdot |^{2}$-cyclically monotone. \\
Then for $\{(x_{i},y_{i})\}_{ i \leq N} \subset \Delta$, since $\Delta \subset \Gamma$,
it holds
\begin{align*} 
\sum_{i=1}^{N} \sfd^{2}(x_{i},y_{i}) = &~ \sum_{i =1}^{N}|\f(x_{i}) - \f(y_{i})|^{2} \crcr
\leq&~ \sum_{i =1}^{N}|\f(x_{i}) - \f(y_{i+1})|^{2} \crcr
\leq &~ \sum_{i=1}^{N} \sfd^{2}(x_{i},y_{i+1}),
\end{align*}
where the last inequality is given by the 1-Lipschitz regularity of $\f$. The claim follows.
\end{proof}

To study the set of branching points is necessary to relate point of branching to geodesics.
In the next Lemma, using Lemma \ref{L:cicli}, we observe that once a branching happens 
there exist two distinct geodesics, both contained in $\Gamma(x)$, 
that are not in relation in the sense of $R$.

\begin{lemma}\label{L:geoingamma}
Let $x \in A_{+}$. 
Then there exist two distinct geodesics $\gamma^{1},\gamma^{2} \in G$ 
such that 
\begin{itemize}
\item[-] $(x,\gamma_{s}^{1}), (x,\gamma_{s}^{2}) \in \Gamma$ for all $s \in [0,1]$;
\item[-] $(\gamma_{s}^{1},\gamma^{2}_{s}) \notin R$ for all $s \in [0,1]$; 
\item[-] $\f(\gamma^{1}_{s}) = \f(\gamma^{2}_{s})$ for all $s \in [0,1]$.
\end{itemize}
Moreover both geodesics are non-constant.
\end{lemma}

\begin{proof}
From the definition of $A_{+}$ there exists $z,w \in \mathcal{T}_{e}$ such that $z,w \in \Gamma(x)$ and $(z,w) \notin R$. 
Since $z,w \in \Gamma(x)$, from Lemma \ref{L:cicli} there exist two geodesics $\gamma^{1},\gamma^{2} \in G$ such that 
$$
\gamma^{1}_{0} = \gamma^{2}_{0} = x, \quad \gamma^{1}_{1} = z, \quad \gamma^{2}_{1} = w.
$$
Since $(z,w) \notin R$, necessarily both $z$ and $w$ are different from $x$ and $x$ is not a final point, that is $x \notin {\mathfrak b}$. 
So the previous geodesics are not constant.
Since $z$ and $w$ can be exchanged, we can also assume that $\f(z) \geq \f(w)$.
Since $z \in \Gamma(x)$, $\f(x) \geq \f(z)$ and by continuity there exists 
$s_{2} \in (0,1]$ such that 
$$
\f(z) = \f(\gamma^{2}_{s_{2}}).
$$
Note that $z \neq \gamma^{2}_{s_{2}}$, otherwise $w \in \Gamma(z)$ and therefore $(z,w) \in R$.
Moreover still $(z,\gamma^{2}_{s_{2}}) \notin R$. Indeed if the contrary was true, then 
$$
0= |\f(z) - \f(\gamma^{2}_{s_{2}}) | = \sfd(z,\gamma^{2}_{s_{2}}),
$$
that is a contradiction with $z \neq \gamma^{2}_{s_{2}}$.

So by continuity there exists $\delta > 0$ such that 
$$
\f (\gamma^{1}_{1-s} ) = \f (\gamma^{2}_{s_{2}(1-s)} ), \qquad \sfd(\gamma^{1}_{1-s}, \gamma^{2}_{s_{2}-s}) > 0,
$$
for all $0 \leq s \leq \delta$. 

Hence reapplying the previous argument $(\gamma^{1}_{1-s}, \gamma^{2}_{s_{2}(1-s)}) \notin R$.
The curve $\gamma^{1}$ and $\gamma^{2}$ of the claim are then obtained 
properly restricting and rescaling the geodesic $\gamma^{1}$ and $\gamma^{2}$ considered so far.
\end{proof}

The previous correspondence between branching points and couples of branching geodesics can be proved to be measurable.
We will make use of the following selection result, Theorem 5.5.2 of \cite{Sri:courseborel}. We again refer to \cite{Sri:courseborel} for some preliminaries on analytic sets.

\begin{theorem}
\label{T:vanneuma}
Let $X$ and $Y$ be Polish spaces, $F \subset X \times Y$ analytic, and $\mathcal{A}$ be the $\sigma$-algebra generated by the analytic subsets of X. 
Then there is an $\mathcal{A}$-measurable section $u : P_1(F) \to Y$ of $F$.
\end{theorem}

Recall that given $F \subset X \times Y$, a \emph{section $u$ of $F$} is 
a function from $P_1(F)$ to $Y$ such that $\textrm{graph}(u) \subset F$.

\begin{lemma}\label{L:selectiongeo}
There exists an $\A$-measurable map $u : A_{+} \mapsto G \times G$ such 
that if $u(x) = (\gamma^{1},\gamma^{2})$ then
\begin{itemize}
\item[-] $(x,\gamma_{s}^{1}), (x,\gamma_{s}^{2}) \in \Gamma$ for all $s \in [0,1]$;
\item[-] $(\gamma_{s}^{1},\gamma^{2}_{s}) \notin R$ for all $s \in [0,1]$; 
\item[-] $\f(\gamma^{1}_{s}) = \f(\gamma^{2}_{s})$ for all $s \in [0,1]$.
\end{itemize}
Moreover both geodesics are non-constant.
\end{lemma}

\begin{proof}
Since $G = \{ \gamma \in \Geo(X) : (\gamma_{0},\gamma_{1}) \in \Gamma \}$
and $\Gamma \subset X \times X$ is closed, the set $G$ is a complete and separable metric space.
Consider now the set 
\begin{align*}
F:= &~ \{ (x,\gamma^{1},\gamma^{2}) \in \mathcal{T}_{e}\times G \times G : (x,\gamma^{1}_{0}), (x,\gamma^{2}_{0}) \in \Gamma \} \crcr
&~ \cap\left( X\times \{ (\gamma^{1},\gamma^{2}) \in G\times G : \sfd(\gamma^{1}_{1},\gamma^{2}_{1})>0 \} \right) \crcr
&~ \cap\left( X\times \{ (\gamma^{1},\gamma^{2}) \in G\times G : \sfd(\gamma^{1}_{0},\gamma^{2}_{0})>0 \} \right) \crcr
&~ \cap\left( X\times \{ (\gamma^{1},\gamma^{2}) \in G\times G : \sfd(\gamma^{1}_{0},\gamma^{1}_{1})>0 \} \right) \crcr
&~ \cap\left( X\times \{ (\gamma^{1},\gamma^{2}) \in G\times G : \f(\gamma^{1}_{i}) = \f(\gamma^{2}_{i}), \, i =0,1 \} \right).
\end{align*}
It follows from Remark \ref{R:regularity} that $F$ is $\sigma$-compact. To avoid possible intersections 
in interior points of $\gamma^{1}$ with $\gamma^{2}$ we consider the following map:
\begin{align*}
h : G \times G &~ \to ~ [0,\infty) \crcr
(\gamma^{1},\gamma^{2}) & ~ \mapsto ~ h(\gamma^{1},\gamma^{2}) : =  \min_{s\in [0,1]} \, \sfd ( \gamma^{1}_{s},\gamma^{2}_{s}). 
\end{align*}
From compactness of $[0,1]$, we deduce the continuity of $h$. Therefore 
$$
\hat F : = F \cap \{ (x,\gamma^{1},\gamma^{2} )\in X \times G\times G :  h(\gamma^{1},\gamma^{2}) > 0\}
$$
is a Borel set and from Lemma \ref{L:geoingamma}, 
$$
\hat F \cap  \left( \{x\} \times G\times G \right) \neq \emptyset
$$
for all $x \in A_{+}$. By Theorem \ref{T:vanneuma} we infer the existence of an $\mathcal{A}$-measurable selection $u$ of $\hat F$.
Since $A_{+} = P_{1}(\hat F)$ and if $u(x) = (\gamma^{1},\gamma^{2})$, then
$$
\sfd( \gamma^{1}_{s},\gamma^{2}_{s}) > 0, \qquad \f(\gamma^{1}_{s}) = \f(\gamma^{2}_{s}), 
$$
for all $s \in [0,1]$, and therefore $(\gamma^{1}_{s},\gamma^{2}_{s}) \notin R$ for all $s \in [0,1]$.
The claim follows.
\end{proof}

We are ready to prove the following

\begin{proposition}\label{P:nobranch}
Let $(X,\sfd,\mm)$ be a m.m.s. such that for any $\mu_{0},\mu_{1} \in \mathcal{P}(X)$ with $\mu_{0} \ll \mm$ any optimal transference plan for $W_{2}$ is concentrated on the 
graph of a function. Then 
$$
\mm(A_{+}) = \mm(A_{-}) = 0.
$$
\end{proposition}

\begin{proof}

{\bf Step 1.} \\ 
Suppose by contradiction that $\mm(A_{+})>0$. 
By definition of $A_{+}$, thanks to Lemma \ref{L:geoingamma} and Lemma \ref{L:selectiongeo}, 
for every $x \in A_{+}$ there exist two non-constant geodesics $\gamma^{1},\gamma^{2} \in G$ 
such that 
\begin{itemize}
\item[-] $(x,\gamma_{s}^{1}), (x,\gamma_{s}^{2}) \in \Gamma$ for all $s \in [0,1]$;
\item[-] $(\gamma_{s}^{1},\gamma^{2}_{s}) \notin R$ for all $s \in [0,1]$; 
\item[-] $\f(\gamma^{1}_{s}) = \f(\gamma^{2}_{s})$ for all $s \in [0,1]$.
\end{itemize}
Moreover the map $A_{+} \ni x \mapsto u(x) : = (\gamma^{1},\gamma^{2}) \in G^{2}$ is $\A$-measurable.

By inner regularity of compact sets (or by Lusin's Theorem), possibly selecting a subset of $A_{+}$ still with strictly positive $\mm$-measure, 
we can assume that the previous map is continuous and in particular the functions
$$
A_{+} \ni x \mapsto \f(\gamma^{i}_{j}) \in \R, \qquad  i =1,2, \ j = 0,1
$$
are all continuous. Put 
$$
\alpha_{x} : = \f(\gamma^{1}_{0}) =  \f(\gamma^{2}_{0}), \qquad \beta_{x} : = \f(\gamma^{1}_{1}) =  \f(\gamma^{2}_{1})  
$$
and note that $\alpha_{x} > \beta_{x}$.
Now we want to show the existence of a subset $B \subset A_{+}$, still with $\mm(B) > 0$, such that 
$$
\sup_{x \in B} \beta_{x} < \inf_{x\in B} \alpha_{x}.
$$
By continuity of $\alpha$ and $\beta$, a set $B$ verifying the previous inequality
can be obtained considering the set $A_{+} \cap B_{r}(x)$, for $x \in A_{+}$ with $r$ sufficiently small.
Since $\mm(A_{+})>0$, for $\mm$-a.e. $x \in A_{+}$ the set $A_{+}\cap B_{r}(x)$ has positive $\mm$-measure.
So the existence of $B \subset A_{+}$ enjoying the aforementioned properties follows.
\medskip

{\bf Step 2.} \\ 
Let $I = [c,d]$ be a non trivial interval such that 
$$
\sup_{x \in B} \beta_{x} < c < d <\inf_{x\in B} \alpha_{x}. 
$$
Then by construction for all $x \in B$ the image of the composition of the geodesics $\gamma^{1}$ and $\gamma^{2}$ with $\f$
contains the interval $I$: 
$$
I \subset \{ \f(\gamma^{i}_{s}) : s \in [0,1] \}, \qquad i = 1,2.
$$
Then fix any point inside $I$, say $c$ and consider for any $x \in B$ the value $s(x)$ such that  $\f(\gamma^{1}_{s(x)}) = \f(\gamma^{2}_{s(x)}) = c$.
We can now define on $B$ two transport maps $T^{1}$ and $T^{2}$ by
$$
B \ni x \mapsto T^{i}(x) : = \gamma^{i}_{s(x)}, \qquad i =1,2.
$$
Accordingly we define the transport plan
$$
\eta : = \frac{1}{2} \left(   (Id, T^{1})_{\sharp} \mm_{B} + (Id, T^{2})_{\sharp} \mm_{B}   \right),
$$
where $\mm_{B} : = \mm(B)^{-1} \mm \llcorner_{B}$. 
\medskip

{\bf Step 3.} \\
The support of $\eta$ is $\sfd^{2}$-cyclically monotone. To prove it we will use Lemma \ref{L:12monotone}. 
The measure $\eta$ is concentrated on the set
\[
\Delta : = \{ (x,\gamma^{1}_{s(x)}) : x \in B \} \cup  \{ (x,\gamma^{2}_{s(x)}) : x \in B \} \subset \Gamma.
\]
Take any two couples $(x_{0},y_{0}), (x_{1},y_{1}) \in \Delta$ and notice that by definition:
\[
\f(y_{1}) - \f(y_{0}) = 0,
\]
and therefore trivially $\left( \f(y_{1}) - \f(y_{0}) \right) \left( \f(x_{1}) - \f(x_{0}) \right) = 0$,
and Lemma \ref{L:12monotone} can be applied to $\Delta$. 
Hence $\eta$ is optimal with $(P_{1})_{\sharp}\eta \ll \mm$ and is not induced by a map; this is a contradiction with 
the assumption. It follows that $\mm(A_{+})=0$. The claim for $A_{-}$ follows in the same manner.
\end{proof}

\begin{remark}\label{R:nobranch} 
If the space is itself non-branching, then Proposition \ref{P:nobranch} can be proved more directly under the assumption \ref{A:1}, that will be introduced  
at the beginning of Section \ref{S:ConditionalMeasures}. Recall that $(X,\sfd,\mm)$ is non-branching if for any $\gamma^{1},\gamma^{2} \in \Geo$ such that 
$$
\gamma^{1}_{0} = \gamma^{2}_{0}, \qquad \gamma^{1}_{t} = \gamma^{2}_{t},
$$
for some $t \in (0,1)$, implies that $\gamma^{1}_{1} = \gamma^{2}_{1}$. In particular the following statement holds

\medskip
\noindent
\emph{ Let $(X,\sfd,\mm)$ be non-branching and assume moreover \ref{A:1} to hold. Then 
$$
\mm(A_{+}) = \mm(A_{-}) = 0.
$$
}
For the proof of this statement (that goes beyond the scope of this note) we refer to \cite{biacava:streconv}, Lemma 5.3.
The same comment will also apply to the next Theorem \ref{T:RCD}.
\end{remark}

To summarize what proved so far introduce also the following notation: the set
\begin{equation}\label{E:transportset}
\mathcal{T} : = \mathcal{T}_{e} \setminus (A_{+} \cup A_{-})
\end{equation}
will be called the \emph{transport set}. Since $\T_{e}, A_{+}$ and $A_{-}$ are $\sigma$-compact sets, 
notice that $\T$ is countable intersection of $\sigma$-compact sets and in particular Borel.

\medskip

\begin{theorem}[Theorem 5.5, \cite{cava:MongeRCD}]\label{T:RCD}
Let $(X,\sfd,\mm)$ be such that for any $\mu_{0},\mu_{1} \in \mathcal{P}(X)$ with $\mu_{0} \ll \mm$ any optimal transference plan for $W_{2}$ is concentrated on the 
graph of a function. Then the set of transport rays $R\subset X \times X$ is an equivalence relation on the transport set $\mathcal{T}$ 
and 
$$
\mm(\mathcal{T}_{e} \setminus \mathcal{T} ) = 0.
$$
\end{theorem}

\bigskip

To recap, we have shown that given a $\sfd$-monotone set 
$\Gamma$,  the set of all those points moved by $\Gamma$, denoted with $\mathcal{T}_{e}$, 
can be written, neglecting a set of $\mm$-measure zero, as the union of a family of disjoint geodesics. 
The next step is to decompose the reference measure $\mm$ restricted to $\mathcal{T}$ 
with respect to the partition given by $R$, where each equivalence class is given by
$$
[x]  =  \{ y \in \mathcal{T}: (x,y) \in R \}.
$$
Denoting the set of equivalence classes with $Q$, we can apply Disintegration Theorem (see Theorem \ref{T:disintegrationgeneral}) 
to the measure space $(\T, \mathcal{B}(\T), \mm)$ and obtain 
the disintegration of $\mm$ consistent with the partition of $\T$ in rays:
$$
\mm\llcorner_{\T} = \int_{Q} \mm_{q} \, \qq(dq), 
$$
where $\qq$ is the quotient measure.

\bigskip


\subsection{Structure of the quotient set}\label{Ss:structure}

In order to use the strength of Disintegration Theorem to localize the measure, one needs to obtain a \emph{strongly consistent} disintegration.  
Following the last part of Theorem \ref{T:disintegrationgeneral}, it is necessary to build a section $S$ of $\T$ together with a measurable quotient map with image $S$.

\begin{proposition}[$Q$ is locally contained in level sets of $\f$]\label{P:Qlevelset}
It is possible to construct a Borel quotient map $\QQ: \T \to Q$ such that the quotient set $Q \subset X$ 
can be written locally  as a level set of $\f$ in the following sense: 
$$
Q = \bigcup_{i\in \N} Q_{i}, \qquad Q_{i} \subset \f^{-1}(\alpha_{i}), 
$$
where $\alpha_i \in \Q$, $Q_{i}$ is  analytic and $Q_{i} \cap Q_{j} = \emptyset$, for $i\neq j$.
\end{proposition}

\begin{proof}
{\bf Step 1.}\\
For each $n \in \N$, consider the set $\T_{n}$ of those points $x$ having ray $R(x)$ longer than $1/n$, i.e.
$$
\T_{n} : = P_{1} \{ (x,z,w) \in \T_{e} \times \T_{e} \times \T_{e} \colon z,w \in R(x), \, \sfd(z,w) \geq 1/n \} \cap \T.
$$
It is easily seen that  $\T=\bigcup_{n \in \N} \T_n$ and that  $\T_{n}$ is Borel: the set $\T_{e}$ is $\sigma$-compact and therefore its projection is again $\sigma$-compact.

Moreover if $x \in \T_{n}, y \in \T$ and $(x,y) \in R$ then also $y \in \T_{n}$: for $x \in \T_{n}$ there exists $z,w \in \T_{e}$ with $z,w\in R(x)$ and $\sfd(z,w)\geq 1/n$. 
Since $x\in \T$ necessarily $z,w \in \T$. Since $R$ is an equivalence relation on $\T$ and $y \in \T$, it follows that $z,w \in R(y)$. Hence $y \in \T_{n}$.
In particular, $\T_{n}$ is the union of all those maximal rays of $\T$ with length at least $1/n$. 

Using the same notation, we have $\T = \cup_{n\in \N} \T_{n}$ with $\T_{n}$ Borel, saturated with respect to $R$, each ray of $\T_{n}$ is longer than $1/n$
and $\T_{n} \cap \T_{n'} = \emptyset$ as soon as $n \neq n'$.

\medskip
Now we consider the following saturated subsets of $\T_{n}$:  for $\alpha \in \Q$
\begin{equation}\label{eq:defTnalpha}
 \T_{n,\alpha}:=  P_{1}  \Big( R \cap \Big \{ (x,y) \in \T_{n} \times \T_{n} \colon  \f(y) = \alpha - \frac{1}{3n}\Big \}  \Big) 
  \cap P_{1} \Big( R \cap \Big \{ (x,y) \in \T_{n} \times \T_{n} \colon  \f(y) = \alpha+  \frac{1}{3n} \Big\}  \Big),
\end{equation}
and we claim  that 
\begin{equation}  \label{eq:Tnalpha}
\T_{n} =   \bigcup_{\alpha \in \Q}  \T_{n,\alpha} . 
\end{equation}

We show the above identity by double inclusion. First note that $(\supset)$ holds trivially. For the converse inclusion $(\subset)$   observe that
for each $\alpha \in \Q$, the set $ \T_{n,\alpha}$ coincides with the family of those rays $R(x) \cap \T_{n}$ such that there exists $y^{+},y^{-} \in R(x)$ such that
\begin{equation}\label{eq:ypm}
\f(y^{+}) = \alpha - \frac{1}{3n}, \qquad \f(y^{-}) = \alpha + \frac{1}{3n}. 
\end{equation}
Then we need to show that any $x \in \T_{n}$, also verifies $x \in  \T_{n,\alpha}$ for a suitable $\alpha \in \Q$. So fix $x \in \T_{n}$ and since $R(x)$ is longer than $1/n$, there exist $z,y^{+},y^{-} \in R(x) \cap \T_{n}$ such that
$$
\f(y^{-}) -\f(z) = \frac{1}{2n}, \qquad  \f(z) -\f(y^{+})= \frac{1}{2n}. 
$$
Consider now the geodesic $\gamma \in G$ such that $\gamma_{0} = y^{-}$ and $\gamma_{1} = y^{+}$. By continuity of 
$[0,1] \ni t \mapsto \f(\gamma_{t})$ it follows the existence of $0 < s_{1}< s_{2} < s_{3} <1$ such that 
$$
\f(\gamma_{s_{3}}) = \f(\gamma_{s_{2}})- \frac{1}{3n}, \qquad \f(\gamma_{s_{1}}) = \f(\gamma_{s_{2}}) + \frac{1}{3n}, \qquad \f  \in \Q.
$$
This concludes the proof of the identity  \eqref{eq:Tnalpha}.  
\medskip

{\bf Step 2.}\\
By the above construction, one can check that for each $\alpha \in \Q$, the  level set $\f^{-1}(\alpha)$ is a quotient set for  $\T_{n,\alpha}$, i.e.
%
%
 $\T_{n,\alpha}$ is formed by disjoint geodesics  each one  intersecting $\f^{-1}(\alpha)$ in exactly one point. Equivalently, $\f^{-1}(\alpha)$ is a section for the partition of $\T_{n}$
induced by $R$. 
 
Moreover $\T_{n,\alpha}$ is obtained as the projection of a Borel set and it is therefore analytic. \\
Since $\T_{n,\alpha}$ is saturated with respect to $R$ either $\T_{n,\alpha} \cap \T_{n,\alpha'} = \emptyset$ or $\T_{n,\alpha} = \T_{n,\alpha'}$. 
Hence, removing the unnecessary $\alpha$, we can assume that 
$\T = \bigcup_{n \in \N, \alpha\in \Q} \T_{n,\alpha}$, is a partition.  
Then we characterize $\QQ : \T \to \T$ defining its graph as follows:
$$
\gr(\QQ) := \bigcup_{n \in \N, \alpha\in \Q} \T_{n,\alpha} \times \left( \f^{-1}(\alpha) \cap\T_{n,\alpha} \right).
$$
Notice that $\gr(\QQ)$ is analytic and therefore $\QQ: \T \to Q$ is Borel (see Theorem 4.5.2 of \cite{Sri:courseborel}). The claim follows.
\end{proof}

\begin{corollary}
The following strongly consistent disintegration formula holds true:
\begin{equation}\label{E:disint}
\mm \llcorner_{\mathcal{T}} = 
			\int_{Q} \mm_{q} \, \qq(dq), \qquad \mm_{q}(\QQ^{-1}(q)) = 1, \ \qq\text{-a.e.}\ q \in Q.
\end{equation}

\end{corollary}

\begin{proof}
From Proposition \ref{P:Qlevelset} there exists an analytic quotient set $Q$ with Borel quotient map $\QQ : \T \to Q$. In particular $Q$ is a section and the push-forward $\sigma$-algebra 
of $\mathcal{B}(\T)$ on $Q$ contains $\mathcal{B}(Q)$.
From Theorem \ref{T:disintr} \eqref{E:disint} follows. 
\end{proof}

\begin{remark}\label{R:regulardisint}
One can improve the regularity of the disintegration formula \eqref{E:disint} as follows.  
From inner regularity of Borel measures there exists $S \subset Q$ $\sigma$-compact, such that 
$\qq(Q \setminus S) = 0$.  The subset $R^{-1}(S) \subset \T$ is again $\sigma$-compact, indeed
\begin{align*}
R^{-1}(S) 
= &~ \{ x\in \T \colon (x,q) \in R, \, q \in S \} = P_{1} (\{ (x,q) \in \T \times S \colon (x,q) \in R \} )  \crcr
= &~ P_{1} (\T \times S \cap R) =  P_{1} (\T_{e} \times S \cap R).
\end{align*}
and the regularity follows. Notice that $R^{-1}(S)$ is formed by non-branching rays and $\mm(\T \setminus R^{-1})(S)) = \qq(Q \setminus S) = 0$. 
Hence we have proved that the transport set with end points $\T_{e}$ admits a saturated, partitioned by disjoint rays, $\sigma$-compact subset of full measure 
with $\sigma$-compact quotient set. 
Since in what follows we will not use the definition \eqref{E:transportset}, we will denote this set with $\T$ and its quotient set with $Q$.
\end{remark}

For ease of notation $X_{q} : = \QQ^{-1}(q)$. The next goal will be to deduce regularity properties for the conditional measures $\mm_{q}$.
The next function will be of some help during the note.

\begin{definition}[Definition 4.5, \cite{biacava:streconv}][Ray map]
\label{D:mongemap}
Define the \emph{ray map}  
$$
g:  \textrm{Dom}(g) \subset Q \times \R \to \mathcal{T}
$$ 
via the formula:
\begin{align*}
\gr (g) : 	= 	&~ \Big\{ (q,t,x) \in Q \times [0,+\infty) \times \mathcal{T}:  (q,x) \in \Gamma, \, \sfd(q,x) = t  \Big\} \crcr
			&~ \cup \Big\{ (q,t,x) \in Q \times (-\infty,0] \times \mathcal{T}  : (x,q) \in \Gamma, \, \sfd(x,q) = t \Big\} \crcr
		=	&~ \gr(g^+) \cup \gr(g^-).
\end{align*}
\end{definition}
Hence the ray map associates to each $q \in Q$ and $t\in  \dom(g(q, \cdot))\subset \R$ the 
unique element $x \in \mathcal{T}$ such that $(q,x) \in \Gamma$ at distance $t$ from $q$
if $t$ is positive or the unique element $x \in \mathcal{T}$ such that $(x,q) \in \Gamma$ at distance $-t$ from $q$
if $t$ is negative. By definition $\textrm{Dom}(g) : = g^{-1}(\mathcal{T})$. 
Notice that from Remark \ref{R:regulardisint} it is not restrictive to assume $\gr (g)$ to be $\sigma$-compact. 
In particular the map $g$ is Borel.

Next we list few (trivial) regularity properties enjoyed by $g$. 
\begin{proposition} \label{P:gammaclass}
The following holds.
\begin{itemize}
\item[-] $g$ is a Borel map.
\item[-] $t \mapsto g(q,t)$ is an isometry and if $s,t \in \dom(g(q,\cdot))$ with $s \leq t$ then $( g(q,s), g(q,t) ) \in \Gamma$;
\item[-] $\textrm{Dom}(g) \ni (q,t) \mapsto g(q,t)$ is bijective on $\QQ^{-1}(Q) = \mathcal{T}$, and its inverse is
$$
x \mapsto g^{-1}(x) = \big( \QQ(x),\pm \sfd(x,\QQ(x)) \big)
$$
where $\QQ$ is the quotient map previously introduced and the positive or negative sign depends on 
$(x,\QQ(x)) \in \Gamma$ or $(\QQ(x),x) \in \Gamma$.
\end{itemize}
\end{proposition}

Observe that from Lemma \ref{L:cicli},  $\dom (g(q,\cdot))$ is a convex subset of $\R$ (i.e. an interval), 
for any $q \in Q$. Using the ray map $g$, we will review in Section \ref{S:ConditionalMeasures} how 
to prove that $\qq$-a.e. conditional measure $\mm_{q}$ is absolutely continuous with respect to the $1$-dimensional Hausdorff measure on $X_{q}$, 
provided $(X,\sfd,\mm)$ enjoys weak curvature properties. The other main use of the ray map $g$ was presented in Section 7 of \cite{biacava:streconv} where it was used 
to build a 1-dimensional metric currents in the sense of Ambrosio-Kirchheim (see \cite{AK}) associated to $\T$.

It is worth also noticing that so far, besides the assumption of Proposition \ref{P:nobranch}, no extra assumption on the geometry of the space was used. 
In particular, given two probability measures $\mu_{0}$ and $\mu_{1}$ with finite first moment, the associated transport set permits 
to decompose the reference measure $\mm$ in one-dimensional conditional measures $\mm_{q}$, i.e. formula \eqref{E:disint} holds.

\bigskip

\subsection{Balanced transportation}\label{Ss:balanced}

Here we want underline that the disintegration (or one-dimensional localization) of $\mm$ induced 
by the $L^{1}$-Optimal Transportation problem between $\mu_{0}$ and $\mu_{1}$ is actually a localization of the Monge problem. 
We will present this fact considering a function $f : X \to \R$ such that 
$$
\int_{X} f(x)\,\mm(dx) =  0, \qquad \int_{X}|f(x)|\sfd(x,x_{0}) \, \mm(dx) < \infty,
$$
and considering $\mu_{0} : = f_{+}\,\mm$ and $\mu_{1} : = f_{-}\,\mm$, where $f_{\pm}$ denotes the positive and the negative part of $f$.
We can also assume $\mu_{0}, \mu_{1} \in \mathcal{P}(X)$ and study the Monge minimization problem between $\mu_{0}$ and $\mu_{1}$.
This setting is equivalent to study the general Monge problem assuming both $\mu_{0},\mu_{1} \ll \mm$; 
note indeed that $\mu_{0}$ and $\mu_{1}$ can always be assumed to be concentrated on disjoint sets (see \cite{biacava:streconv} for details).

If $\f$ is an associated Kantorovich potential producing as before the transport set $\T$, we have a disintegration of $\mm$ as follows: 
$$
\mm\llcorner_{\T} = \int_{Q} \mm_{q} \, \qq(dq), \qquad \mm_{q}(X_{q}) =1,\  \qq\textrm{-a.e.} \, q \in Q.
$$
Then the natural localization of the Monge problem would be to consider for every $q \in Q$ the Monge minimization problem between 
$$
\mu_{0\, q} : = f_{+} \,\mm_{q},\quad  \mu_{1\, q} : = f_{-} \,\mm_{q},
$$
in the metric space $(X_{q},\sfd)$ (that is isometric via the ray map $g$ to an interval of $\R$ with the Euclidean distance). 
To check that this family of problems makes sense we need to prove the following

\begin{lemma}\label{L:balancing}
It holds that for $\qq$-a.e. $q \in Q$ one has $\int_{X} f \, \mm_{q} = 0$.  
\end{lemma}

\begin{proof}
Since for both $\mu_{0}$ and $\mu_{1}$ the set $\mathcal{T}_{e} \setminus \mathcal{T}$ is negligible ($\mu_{0},\mu_{1} \ll \mm$),
for any Borel set $C \subset Q$ 
\begin{eqnarray}\label{E:mu0=mu1}
\mu_{0}(\QQ^{-1}(C)) 	&= & \pi \Big(  (\QQ^{-1}(C) \times X)  \cap \Gamma \setminus \{ x = y\} \Big) \nonumber \\
					&= & \pi \Big(  ( X \times \QQ^{-1}(C))  \cap \Gamma \setminus \{ x = y\} \Big) \nonumber \\
					&= &  \mu_{1}(\QQ^{-1}(C)),  \label{eq:mu0=mu1}
\end{eqnarray}
where the second equality follows from the fact that $\mathcal{T}$ does not branch: indeed since $\mu_{0}(\mathcal{T}) = \mu_{1}(\mathcal{T}) = 1$,  
then $\pi \big( (\Gamma \setminus \{ x= y\}) \cap \mathcal{T} \times \mathcal{T}  \big) =1$ and therefore if $x,y \in \mathcal{T}$ and $(x,y)\in \Gamma$, 
then necessarily $\QQ(x) = \QQ(y)$, that is they belong to the same ray. It follows that 
$$
(\QQ^{-1}(C) \times X)  \cap (\Gamma \setminus \{ x = y\}) \cap (\mathcal{T} \times \mathcal{T}) = 
( X\times \QQ^{-1}(C) )  \cap (\Gamma \setminus \{ x = y\}) \cap (\mathcal{T} \times \mathcal{T}),
$$
and \eqref{E:mu0=mu1} follows.

Since $f$ has null mean value it holds $\int_X f_{+}(x) \mm(dx)= - \int_X f_{-}(x) \mm(dx)$, 
which combined with \eqref{E:mu0=mu1} implies that for each Borel $C \subset Q$ 
\begin{align*}
\int_{C}  \int_{X_{q}} f(x) \mm_{q}(dx) \qq(dq) 	= &~ \int_{C}  \int_{X_{q}} f_{+}(x) \mm_{q}(dx) \qq(dq) - \int_{C}  \int_{X_{q}} f_{-}(x) \mm_{q}(dx) \qq(dq) \crcr
									= &~ \left( \int_{X} f_{+}(x) \mm(dx) \right)^{-1} \left( \mu_{0}(\QQ^{-1}(C)) - \mu_{1}(\QQ^{-1}(C))  \right)  \crcr 
									= &~ 0.
\end{align*}
Therefore for $\qq$-a.e. $q \in Q$ the integral $\int f \, \mm_{q}$ vanishes and the claim follows.
\end{proof}

It can be proven in greater generality and without assuming $\mu_{1}\ll \mm$ that the Monge problem is localized once 
a strongly consistent disintegration of $\mm$ restricted to the transport ray is obtained. See \cite{biacava:streconv} for details.

\bigskip
\bigskip


\section{Regularity of conditional measures} \label{S:ConditionalMeasures}

We now review regularity and curvature properties of $\mm_{q}$. What contained in this section is a collection of results spread across 
\cite{biacava:streconv, cava:decomposition, cava:MongeRCD} and \cite{CM1}. We try here to give a unified presentation.
We will inspect three increasing level of regularity: for $\qq$-a.e. $q \in Q$
\begin{enumerate}[label=(\textbf{R.\arabic*})]
\item \label{R:1}  $\mm_{q}$ has no atomic part, i.e. $\mm_{q}(\{x\}) = 0$, for any $x \in X_{q}$;  \medskip
\item \label{R:2}  $\mm_{q}$ is absolutely continuous with respect to $\mathcal{H}^{1}\llcorner_{X_{q}} = g(q,\cdot)_{\sharp} \L^{1}$; \medskip
\item \label{R:3}  $\mm_{q} = g(q,\cdot)_{\sharp} (h_{q}\,\L^{1})$ verifies $\CD(K,N)$, i.e. the m.m.s. $(\R, |\cdot|, h_{q} \, \L^{1})$ verifies $\CD(K,N)$.\medskip
\end{enumerate}
We will review how to obtain \ref{R:1}, \ref{R:2}, \ref{R:3} starting from the following three \emph{increasing} regularity assumptions on the space: 
\begin{enumerate}[label=(\textbf{A.\arabic*})]
\item \label{A:1} if $C \subset \T$ is compact with $\mm(C)> 0$, then $\mm(C_{t}) > 0$ for uncountably many $t \in \R$; \medskip
\item \label{A:2} if $C \subset \T$ is compact with $\mm(C)> 0$, then $\mm(C_{t}) > 0$ for a set of $t \in \R$ with $\L^{1}$-positive measure;  \medskip
\item \label{A:3} the m.m.s. $(X,\sfd,\mm)$ verifies $\CD(K,N)$. \medskip
\end{enumerate}

Given a compact set $C \subset X$, we indicate with $C_{t}$ its translation along the transport set at distance with sign $t$, see the following 
Definition \ref{D:evolution}.

We will see that:  \ref{A:1} implies \ref{R:1}, \ref{A:2} implies \ref{R:2} and \ref{A:3} implies \ref{R:3}.
Actually we will also show a variant of $\ref{A:3}$ (assuming $\MCP$ instead of $\CD$) implies a variant of \ref{R:3} ($\MCP$ instead of $\CD$).

Even if we do not to state it each single time, assumptions \ref{A:1} and \ref{A:2} are not hypothesis on the smoothness of the space but 
on the regularity of the set $\Gamma$ and therefore on the Monge problem itself; they should both be read as: 
\emph{for $\mu_{0}$ and $\mu_{1}$ probability measures over $X$, assume the existence of a 1-Lipschitz Kantorovich potential $\f$ such that 
the associated transport set $\T$ verifies} \ref{A:1} \emph{(or }\ref{A:2}\emph{)}.

\bigskip

\subsection{Atomless conditional probabilities}

The results presented here are taken from \cite{biacava:streconv}.

\begin{definition}\label{D:evolution}
Let $C \subset \T$ be a compact set. For $t \in \R$ define the \emph{$t$-translation $C_{t}$ of $C$} by
$$
C_t := g \big( \{ (q,s +t) \colon (q,s) \in g^{-1}(C) \} \big).
$$
\end{definition}
\noindent
Since $C \subset \T$ is compact, $g^{-1}(C)\subset Q \times \R$ is $\sigma$-compact  ($\gr (g)$ is $\sigma$-compact) and the same holds true for
$$
\{ (q,s +t) \colon (q,s) \in g^{-1}(C) \}. 
$$
Since
$$
C_{t} = P_{3}( \gr(g) \cap \{ (q,s +t) \colon (q,s) \in g^{-1}(C) \} \times \T  ),
$$
it follows that $C_{t}$ $\sigma$-compact (projection of $\sigma$-compact sets is again $\sigma$-compact). \\
Moreover the set $B : = \{ (t,x) \in \R \times \T \colon x \in C_{t} \}$ is Borel and therefore by Fubini's Theorem 
the map  $t \mapsto \mm(C_t)$ is Borel. It follows that \ref{A:1} makes sense.

\begin{proposition}[Proposition 5.4, \cite{biacava:streconv}]\label{P:nonatoms}
Assume \emph{\ref{A:1}} to hold and the space to be non-branching. 
Then \emph{\ref{R:1}} holds true, that is for $\qq$-a.e. $q \in Q$ the conditional measure $\mm_{q}$ has no atoms.
\end{proposition}

\begin{proof}
The partition in trasport rays and the associated disintegration are well defined, see Remark \ref{R:nobranch}.
From the regularity of the disintegration and the fact that $\qq(Q) = 1$, 
we can assume that the map $q \mapsto \mm_q$ is weakly continuous on a compact set $K \subset Q$ with $\qq(Q \setminus K) < \ve$ such 
that the length of the ray $X_{q}$, denoted by $L(X_{q})$, is strictly larger than $\ve$ for all $q \in K$. It is enough to prove the proposition on $K$.
\medskip

{\bf Step 1.}\\ 
From the continuity of $K \ni q \mapsto \mm_q \in \mathcal{P}(X)$ w.r.t. the weak topology, it follows that the map
$$
q \mapsto C(q) := \big\{ x \in X_{q}: \mm_q(\{x\}) > 0 \big\} = \cup_n \big\{ x \in X_{q}: \mm_q(\{x\}) \geq 2^{-n} \big\}
$$
is $\sigma$-closed, i.e. its graph is countable union of closed sets: 
in fact, if $(q_m,x_m) \to (y,x)$ and $\mm_{q_m}(\{x_m\}) \geq 2^{-n}$, then $\mm_q(\{x\}) \geq 2^{-n}$ 
by upper semi-continuity on compact sets.

Hence it is Borel, and by Lusin Theorem (Theorem 5.8.11 of \cite{Sri:courseborel}) it is the countable union of Borel graphs: 
setting in case $c_i(q) = 0$, we can consider them as Borel functions on $K$ and order them w.r.t. $\Gamma$ in the following sense:
$$
\mm_{q,\textrm{atomic}} = \sum_{i \in \Z} c_i(q) \delta_{x_i(q)}, \quad (x_i(q), x_{i+1}(q)) \in \Gamma, \ i \in \Z,
$$
with $K \ni q \mapsto x_{i}(q)$ Borel.
\medskip

{\bf Step 2.} \\
Define the sets
$$
S_{ij}(t) := \Big\{ q \in K: x_i(q) = g \big( g^{-1}(x_j(q)) + t \big) \Big\},
$$
Since $K \subset Q$, to define $S_{ij}(t)$ we are using the $\gr (g) \cap Q \times \R \times \mathcal{T}$, which is $\sigma$-compact: 
hence $\gr(S_{ij})$ is analytic.
For $A_j := \{x_j(q), q \in K\}$ and $t \in \R^+$ we have that
\begin{align*}
\mm((A_j)_t) =&~ \int_K \mm_q((A_j)_t)\, \qq(dq) = \int_K \mm_{q,\textrm{atomic}}((A_j)_t) \, \qq(dq) \crcr
=&~ \sum_{i \in \Z} \int_K c_i(q) \delta_{x_i(q)} \big( g(g^{-1}(x_j(q)) + t) \big)\, \qq(dq) = \sum_{i \in \Z} \int_{S_{ij}(t)} c_i(q) \, \qq(dq),
\end{align*}
and we have used that $A_j \cap X_{q}$ is a singleton. Then for fixed $i,j \in \N$, again from the fact that $A_j \cap X_{q}$ is a singleton
$$
S_{ij}(t) \cap S_{ij}(t') =
\begin{cases}
S_{ij}(t) & t = t', \crcr
\emptyset & t \not= t',
\end{cases}
$$
and therefore the cardinality of the set $\big\{ t : \qq(S_{ij}(t)) > 0 \big\}$ has to be countable.
On the other hand,
$$
\mm((A_j)_t) > 0 \quad \Longrightarrow \quad t \in \bigcup_i \big\{ t : \qq(S_{ij}(t)) > 0 \big\},
$$
contradicting \ref{A:1}.
\end{proof}

\bigskip

\subsection{Absolute continuity}
\label{Ss:regolarita'}

The results presented here are taken from \cite{biacava:streconv}.
The condition \ref{A:2} can be stated also in the following way: for every compact set $C \subset \T$
$$
\mm(C) > 0 \quad \Longrightarrow \quad \int_{\R} \mm(C_t) dt > 0.
$$

\begin{lemma}
\label{Lem:dec}
Let $\mm$ be a Radon measure and
$$
\mm_q = r_{q}\, g(q,\cdot)_\sharp \mathcal{L}^1 + \omega_q, \quad \omega_q \perp g(q,\cdot)_\sharp \mathcal{L}^1
$$
be the Radon-Nikodym decomposition of $\mm_q$ w.r.t. $g(q,\cdot)_\sharp \mathcal{L}^1$. Then there exists a Borel set $C\subset X$ such that
$$
\mathcal{L}^{1} \Big(P_{2}\big( g^{-1} (C) \cap (\{q\} \times \R) ) \big) \Big)= 0,
$$
and $\omega_q = \mm_q \llcorner_C$ for $\qq$-a.e. $q \in Q$.
\end{lemma}

\begin{proof}
Consider the measure $\lambda = g_\sharp (\qq \otimes \mathcal L^1)$,
and compute the Radon-Nikodym decomposition
\[
\mm  = \frac{D \mm}{D \lambda} \lambda + \omega.
\]
Then there exists a Borel set $C$ such that $\omega = \mm \llcorner_C$ and $\lambda(C)=0$. The set $C$ proves the Lemma. Indeed 
$C = \cup_{q \in Q} C_{q}$ where $C_{q} = C \cap R(q)$ is such that 
$\mm_q \llcorner_{C_{q}} = \omega_{q} $ and $g(q,\cdot)_{\sharp}\mathcal{L}^{1}(C_{q})=0$ for $\qq$-a.e. $q \in Q$.
\end{proof}

\begin{theorem}[Theorem 5.7, \cite{biacava:streconv}]\label{teo:a.c.}
Assume \emph{\ref{A:2}} to hold and the space to be non-branching. Then \emph{\ref{R:2}} holds true, that is for $\qq$-a.e. $q \in Q$ the conditional measure $\mm_{q}$ 
is absolute continuous with respect to $g(q,\cdot)_{\sharp}\L^{1}$.
\end{theorem}

The proof is based on the following simple observation.

\medskip
\noindent Let $\eta$ be a Radon measure on $\erre$. Suppose that for all $A \subset \erre$ Borel with $\eta(A)>0$ it holds
\[
\int_{\R^+} \eta(A+t) dt = \eta \otimes\mathcal{L}^1 \big( \{ (x,t): t \geq 0, x - t \in A \} \big) > 0.
\]
Then $\eta \ll \mathcal{L}^1$.

\begin{proof}
The proof will use Lemma \ref{Lem:dec}: take $C$ the set constructed in Lemma \ref{Lem:dec} and suppose by contradiction that
$$
\mm(C) > 0 \quad \text{and} \quad \qq \otimes \mathcal{L}^1 (g^{-1}(C)) = 0.
$$
In particular, for all $t \in \R$ it follows that
$$
\qq \otimes \mathcal{L}^1 (g^{-1}(C_t)) = 0.
$$
By Fubini-Tonelli Theorem
\begin{align*}
0< 	&~ \int_{\R^+} \mm(C_t) \,dt  =  \int_{\R^+} \bigg( \int_{g^{-1}(C_t)} (g^{-1})_\sharp \mm(dq\,d\tau) \bigg) dt \crcr
=	&~ \big( (g^{-1})_\sharp \mm \otimes \mathcal{L}^1 \big) \Big( \Big\{ (q,\tau,t): (q,\tau) \in g^{-1}(\mathcal{T}), (q,\tau-t) \in g^{-1}(C) \Big\} \Big) \crcr
\leq	&~ \int_{Q \times \R} \mathcal{L}^1 \big( \big\{\tau - g^{-1}(C \cap \QQ^{-1}(q)) \big\} \big) \, (g^{-1})_{\sharp} \mm (dq\, d\tau) \crcr
=	&~ \int_{Q \times \R} \mathcal{L}^1 \big( g^{-1}(C \cap \QQ^{-1}(q)) \big) \,(g^{-1})_{\sharp} \mm (dq\,d\tau) \crcr
=	&~ \int_{Q} \mathcal{L}^1 \big( g^{-1}(C \cap \QQ^{-1}(y)) \big)\, \qq(dy) = 0.
\end{align*}
That gives a contradiction.
\end{proof}

The proof of Theorem \ref{teo:a.c.} inspired the definition of \emph{inversion points} and of \emph{inversion plan} as presented in \cite{CM0}, in particular see 
{\bf Step 2.} of the proof of Theorem 5.3 of \cite{CM0}.

\bigskip

\subsection{Weak Ricci curvature bounds: $\MCP(K,N)$}

The presentation of the following results is taken from \cite{cava:MongeRCD}. The same results were already proved in \cite{biacava:streconv}
using more involved arguments and different notation.

In this section we additionally assume the metric measure space to satisfy the measure contraction property $\MCP(K,N)$. Recall that the space 
is also assumed to be non-branching.

\begin{lemma}\label{L:evo1}
For each Borel $C \subset \mathcal{T}$ and $\delta \in \R$ the set 
$$
\left( C \times \{ \f= \delta \} \right) \cap \Gamma,
$$
is $\sfd^{2}$-cyclically monotone.
\end{lemma}

\begin{proof} The proof follows straightforwardly from Lemma \ref{L:12monotone}. The set 
$\left( C \times \{ \f = c \} \right) \cap \Gamma$ is trivially a subset of $\Gamma$ and whenever  
$$
(x_{0},y_{0}), (x_{1},y_{1}) \in \left( C \times \{ \f = \delta \} \right) \cap \Gamma,
$$
then $(\f (y_{1}) -  \f(y_{0})  ) \cdot (\f(x_{1}) -  \f(x_{0})  ) = 0$.
\end{proof}

We can deduce the following

\begin{corollary}\label{C:evo1}
For each Borel $C \subset \mathcal{T}$ and $\delta \in \R$ define 
$$
C_{\delta}: =P_{1}(\left( C \times \{ \f = \delta \} \right) \cap \Gamma).
$$
If $\mm(C_{\delta}) > 0$, there exists a unique $\nu \in \Opt$ such that 
\begin{equation}\label{E:12mappa}
\left( e_{0} \right)_{\sharp} \nu = \mm( C_{\delta} )^{-1} \mm\llcorner_{C_{\delta}},  
\qquad (e_{0},e_{1})_{\sharp}( \nu ) \Big(   \left( C \times \{ \f = \delta \} \right) \cap \Gamma \Big) = 1.
\end{equation}
\end{corollary}

\bigskip
\noindent
From Corollary \ref{C:evo1}, we infer the existence of a map $T_{C,\delta}$ depending on $C$ and $\delta$ such that 
$$
\left( Id,T_{C,\delta} \right)_{\sharp} \left( \mm( C_{\delta} )^{-1} \mm\llcorner_{C_{\delta}} \right) = (e_{0},e_{1})_{\sharp} \nu.
$$
Taking advantage of the ray map $g$, we define a convex combination between the identity map and 
$T_{C,\delta}$ as follows:
$$
C_{\delta} \ni x \mapsto \left(T_{C,\delta}\right)_{t}(x) \in \{ z \in \Gamma(x) : \sfd(x,z) = t \cdot \sfd(x,T_{C,\delta}(x)) \}.
$$
Since $C \subset \mathcal{T}$, the map $\left(T_{C,\delta}\right)_{t}$ is well defined for all $t\in [0,1]$.
We then define the evolution of any subset $A$ of $C_{\delta}$ in the following way:
$$
[0,1] \ni t \mapsto \left(T_{C,\delta}\right)_{t}(A).
$$
In particular from now on we will adopt the following notation:  
$$
A_{t} : = \left(T_{C,\delta} \right)_{t}(A), \qquad \forall A \subset C_{\delta}, \ A \ \textrm{ compact}. 
$$
So for any Borel $C \subset \mathcal{T}$ compact and  $\delta \in \R$ we have defined an evolution 
for compact subsets of $C_{\delta}$. The definition of the evolution depends both on $C$ and $\delta$.

\begin{remark}\label{R:regularity2}
Here we spend a few lines on the measurability of the maps involved in the definition of evolution of sets assuming for simplicity $C$ to be compact.
First note that since $\Gamma$ is closed and $C$ is compact, we can prove that also $C_{\delta}$ is compact.
Indeed from compactness of $C$ we obtain that $\f$ is bounded on $C$ and then, since $C$ is bounded, it follows that also
$C \times \{\f = c \} \cap \Gamma$ is bounded. Since $X$ is proper, compactness follows.
Moreover 
$$
\gr (T_{C,\delta}) = \left( C \times \{ \f= \delta \} \right) \cap \Gamma,
$$
hence $T_{C,\delta}$ is continuous. Moreover 
$$
\left(T_{C,\delta}\right)_{t}(A) = P_{2} \left( \{(x,z) \in \Gamma \cap (A \times X) : \sfd(x,z) = t\cdot \sfd(x,T_{C,\delta}(x))  \}\right),
$$
hence if $A$ is compact, the same holds for $\left(T_{C,\delta}\right)_{t}(A)$ and 
$$
[0,1] \ni t \mapsto \mm(\left(T_{C,\delta}\right)_{t}(A)) 
$$
is $\mm$-measurable. 
\end{remark}

The next result gives  quantitative information on the behavior of the map $t \mapsto \mm(A_{t})$. 
The statement will be given assuming the lower bound on the generalized Ricci curvature $K$ to be positive. 
Analogous estimates holds for any $K \in \R$.

\begin{proposition}\label{P:mcp}
For each compact $C \subset \mathcal{T}$ and $\delta \in \R$ such that $\mm(C_{\delta}) >0$, it holds
\begin{equation}\label{E:mcp}
\mm(  A_{t}  ) \geq  (1-t) \cdot \inf_{x\in A}   \left(\frac{\sin \left( (1-t) \sfd(x,T_{C,\delta}(x) )\sqrt{K/(N-1)}   \right) }
{\sin \left( \sfd(x,T_{C,\delta}(x))\sqrt{K/(N-1)}   \right)} \right)^{N-1}  \mm(A), 
\end{equation}
for all $t \in [0,1]$ and $A \subset C_{\delta}$ compact set.
\end{proposition}

\begin{proof}

The proof of \eqref{E:mcp} is obtained by the standard method of approximation with Dirac deltas of the second marginal.
Even though similar arguments already appeared many times in literature, in order to be self-contained, we include all the details. 
For ease of notation $T = T_{C,\delta}$ and $C = C_{\delta}$.
\medskip

{\bf Step 1.}\\
Consider a sequence $\{ y_{i} \}_{i \in \enne} \subset \{\f = \delta\}$ dense in $T(C)$.
For each $I \in \N$, define the family of sets 
$$
E_{i,I} : = \{ x \in C : \sfd(x,y_{i}) \leq  \sfd(x,y_{j}) , j =1,\dots, I \},
$$
for $i =1, \dots, I$.  Then for all $I \in \N$, by the same argument of Lemma \ref{L:evo1}, the set 
$$
\Lambda_{I}: = \bigcup_{i =1}^{I} E_{i,I}\times \{ y_{i} \} \subset X \times X, 
$$
is $\sfd^{2}$-cyclically monotone. Consider then $A_{i,I} : = A \cap E_{i,I}$ and the approximate evolution 
$$
A_{i,I,t} : = \{ z \in X \colon \sfd(z,y_{i}) = (1-t)\sfd(x,y_{i}), \ x \in A_{i,I}\};
$$
and notice that $A_{i,I,0} = A_{i,I}$. Then by $\MCP(K,N)$ it holds
$$
\mm(  A_{i,I,t}  ) \geq  (1-t) \cdot \inf_{x\in A_{i,I}}   \left(\frac{\sin \left( (1-t) \sfd(x,x_{i} )\sqrt{K/(N-1)}   \right) }
{\sin \left( \sfd(x,x_{i})\sqrt{K/(N-1)}   \right)} \right)^{N-1}  \mm(A_{i,I}).
$$
Taking the sum over $i \leq I$ in the previous inequality implies 
$$
\sum_{i \leq I} \mm(  A_{i,I,t}  ) \geq  (1-t) \cdot \inf_{x\in A}   \left(\frac{\sin \left( (1-t) \sfd(x,T_{I}(x) )\sqrt{K/(N-1)}   \right) }
{\sin \left( \sfd(x,T_{I}(x))\sqrt{K/(N-1)}   \right)} \right)^{N-1}  \mm(A),
$$
where $T_{I}(x) : = y_{i}$ for $x \in E_{i,I}$. 
From $\sfd^{2}$-cyclically monotonicity and the non-branching of the space, up to a set of measure zero, the map $T_{I}$ is well defined, 
i.e. $\mm(E_{i,I} \cap E_{j,I}) = 0$ for $i \neq j$. 
It follows that for each $I \in \N$ we can remove a set of measure zero from $A$ and obtain
$$
A_{i,I,t} \cap A_{j,I,t} = \emptyset, \quad i\neq j. 
$$
As before consider also the interpolated map $T_{I,t}$ and observe that  $A_{I,t} = T_{I,t} (A)$. Since also $A$ is compact we obtain
$$
\mm(  A_{I,t}  ) \geq  (1-t) \cdot \min_{x\in A}   \left(\frac{\sin \left( (1-t) \sfd(x,T_{I}(x) )\sqrt{K/(N-1)}   \right) }
{\sin \left( \sfd(x,T_{I}(x))\sqrt{K/(N-1)}   \right)} \right)^{N-1}  \mm(A).
$$
\medskip

{\bf Step 2.}\\ 
Since $C$ is a compact set, for every $I \in \N$ 
the set $\Lambda_{I}$ is compact as well and it is a subset of $C \times \{ \f = \delta \}$ that 
can be assumed to be compact as well. By compactness, there exists a subsequence $I_{n}$ and 
a compact set  $\Theta \subset  C \times \{ \f = \delta \}$ compact such that 
$$
\lim_{n \to \infty} \sfd_{\mathcal{H}}(\Lambda_{I_{n}}, \Theta) = 0,
$$
where $\sfd_{\mathcal{H}}$ is the Hausdorff distance.
Since the sequence $\{y_{i}\}_{i\in \enne}$ is dense in $\{\f = \delta \}$ and $C \subset \T$ is compact, by definition of $E_{i,I}$,
necessarily for every $(x,y) \in \Theta$ it holds
$$
\f(x)+ \f(y) = \sfd(x,y), \quad \f(y) = \delta.
$$
Hence $\Theta \subset \Gamma \cap C\times \{\f = \delta\}$ and this in particular implies, by upper semicontinuity of $\mm$ 
along converging sequences of closed sets, that 
$$
\mm(A_{t}) \geq \limsup_{n \to \infty} \mm(A_{I_{n},t})\, .
$$
The claim follows.
\end{proof}

As the goal is to localize curvature conditions, we first need to prove that almost every conditional probability is absolutely continuous 
with respect to the one dimensional Hausdorff measure restricted to the correct geodesic.
One way is to prove that Proposition \ref{P:mcp} implies \ref{A:2} and then apply Theorem \ref{teo:a.c.} to obtain \ref{R:2} (approach used in \cite{biacava:streconv}).
Another option is to repeat verbatim the proof of Theorem \ref{teo:a.c.} substituting the translation with the evolution considered in 
Proposition \ref{P:mcp} and to observe that the claim follows (approach used in \cite{cava:MongeRCD}). 
So we take for granted the following 

\begin{proposition}\label{P:MCP-1}
Assume the non-branching m.m.s. $(X,\sfd,\mm)$ to satisfy $\MCP(K,N)$. 
Then \emph{\ref{R:2}} holds true, that is for $\qq$-a.e. $q \in Q$ the conditional measure $\mm_{q}$ 
is absolute continuous with respect to $g(q,\cdot)_{\sharp}\L^{1}$.
\end{proposition}

To fix the notation, we now have proved the existence of a Borel function $h : \dom(g) \to \R_{+}$ such that 
\begin{equation}\label{E:definitionh}
\mm\llcorner\T = g_{\sharp} \left( h \, \qq \otimes \L^{1} \right)
\end{equation}
Using standard arguments, estimate \eqref{E:mcp} can be localized at the level of the density $h$: 
for each compact set $A \subset \mathcal{T}$ 
\begin{align*}
\int_{P_{2}(g^{-1}(A_{t}) )} & h(q,s) \mathcal{L}^{1}(ds)  \crcr
 \geq (1-t) & \left( \inf_{\tau \in P_{2}(g^{-1}(A))} \frac{\sin( (1-t) |\tau - \sigma|  \sqrt{K/(N-1)} )   }{\sin( |\tau - \sigma|  \sqrt{K/(N-1)} )}  \right)^{N-1}
					\int_{P_{2}(g^{-1}(A))} h(q,s) \mathcal{L}^{1}(ds),
\end{align*}
for $\qq$-a.e. $q \in Q$ such that $g(q,\sigma) \in \mathcal{T}$. Then using change of variable, one obtains that for $\qq$-a.e. $q \in Q$:
$$
h(q,s+|s-\sigma| t ) \geq \left(
\frac{\sin( (1-t) |s - \sigma|  \sqrt{K/(N-1)} )   }{\sin( |s - \sigma|  \sqrt{K/(N-1)} )}  \right)^{N-1}
 h(y,s),
$$
for $\mathcal{L}^{1}$-a.e.  $s \in P_{2}(g^{-1}(R(q)))$ and $\sigma \in \R$ such that $s + |\sigma -s| \in P_{2}(g^{-1}(R(q)))$.
We can rewrite the estimate in the following way: 
$$
h(q, \tau ) \geq \left(
\frac{\sin(  ( \sigma - \tau )  \sqrt{K/(N-1)} )   }{\sin( ( \sigma - s )  \sqrt{K/(N-1)} )}  \right)^{N-1} h(q,s),
$$
for $\mathcal{L}^{1}$-a.e. $s \leq \tau \leq  \sigma$ such that $g(q,s), g(q,\tau), g(q,\sigma) \in \mathcal{T}$. 
Since evolution can be also considered backwardly, we have proved the next

\begin{theorem}[Localization of $\MCP$, Theorem 9.5 of \cite{biacava:streconv}]\label{T:densityestimates}
Assume the non-branching m.m.s. $(X,\sfd,\mm)$ to satisfy $\MCP(K,N)$. 
For $\qq$-a.e. $q \in Q$ it holds: 
$$
\left( \frac{\sin(  ( \sigma_{+} - \tau )  \sqrt{K/(N-1)} )   }{\sin( ( \sigma_{+} - s )  \sqrt{K/(N-1)} )}  \right)^{N-1} 
\leq \frac{h(q, \tau )} {h(q,s)} 
\leq  \left( \frac{\sin(  (  \tau - \sigma_{-} )  \sqrt{K/(N-1)} )   }{\sin( (s - \sigma_{-}  )  \sqrt{K/(N-1)} )}  \right)^{N-1},
$$
for $\sigma_{-} < s \leq \tau < \sigma _{+}$ such that their image via $g(q,\cdot)$ is contained in $R(q)$.
\end{theorem} 

\noindent
In particular from Theorem \ref{T:densityestimates} we deduce that 
\begin{equation}\label{E:regularityh}
\{ t \in \dom(g(q,\cdot))  \colon  h(q,t) > 0 \} = \dom(g(q,\cdot)), 
\end{equation}
in particular such set is convex and $t \mapsto h(q,t)$ is locally Lipschitz continuous.

\bigskip

\subsection{Weak Ricci curvature bounds: $\CD(K,N)$}

The results presented here are taken from \cite{CM1}.

We now turn to proving that the conditional probabilities inherit the synthetic Ricci curvature lower bounds, that is, \ref{A:3} implies \ref{R:3}.
Actually it is enough to assume the space to verify such a lower bound only locally to obtain 
globally the synthetic Ricci curvature lower bound on almost every 1-dimensional metric measure spaces.

Since under the essentially non-branching condition  $\CD_{loc}(K,N)$ implies $\MCP(K,N)$ and existence and uniqueness of optimal transport maps, 
see \cite{cavasturm:MCP}, we can already assume \eqref{E:definitionh} and 
\eqref{E:regularityh} to hold. In particular $t \mapsto h_{q}(t)$ is locally Lipschitz continuous, where, for easy of notation $h_{q} = h(q,\cdot)$.

\begin{theorem}[Theorem 4.2 of \cite{CM1}]\label{T:CDKN-1}
Let $(X,\sfd,\mm)$ be an essentially non-branching m.m.s. verifying the $\CD_{loc}(K,N)$ condition for some $K\in \R$ and $N\in [1,\infty)$.

Then for any 1-Lipschitz function  $\f:X\to \R$, the associated transport set $\Gamma$ induces a disintegration 
of $\mm$ restricted to the transport set verifying the following inequality: if $N> 1$ 
\medskip

for $\qq$-a.e. $q \in Q$ the following curvature inequality holds:  
\begin{equation}\label{E:curvdensmm}
h_{q}( (1-s)  t_{0}  + s t_{1} )^{1/(N-1)}  
 \geq \sigma^{(1-s)}_{K,N-1}(t_{1} - t_{0}) h_{q} (t_{0})^{1/(N-1)} + \sigma^{(s)}_{K,N-1}(t_{1} - t_{0}) h_{q} (t_{1})^{1/(N-1)},
\end{equation}
for all $s\in [0,1]$ and for all $t_{0}, t_{1} \in \dom(g(q,\cdot))$ with  $t_{0} < t_{1}$. If $N =1$, for $\qq$-a.e. $q \in Q$ the density $h_{q}$ is constant.
\end{theorem}

\begin{proof}
We first consider the case $N>1$.

{\bf Step 1.} \\
Thanks to Proposition \ref{P:Qlevelset}, without any loss of generality we can assume that the quotient set $Q$ (identified with the set $\{g(q,0) : q \in Q\}$) 
is locally a subset of a level set of the map $\f$ inducing the transport set, i.e.
there exists a countable partition $\{ Q_{i}\}_{i\in \N}$ with $Q_{i} \subset Q$ Borel set such that
$$
\{ g(q,0) : q \in Q_{i} \} \subset \{ x \in X : \f(x) = \alpha_{i} \}.
$$
It is clearly sufficient to prove  \eqref{E:curvdensmm} on each $Q_{i}$;  so  fix $\bar i \in \N$ and for ease of notation assume $\alpha_{\bar i} = 0$ and $Q = Q_{\bar i}$.
As $\dom(g(q,\cdot))$ is a convex subset of $\R$, we can also restrict to a uniform subinterval 
$$
 (a_0,a_1) \subset \dom(g(q,\cdot)), \qquad \forall \ q \ \in Q_{i},
$$
for some $a_0,a_1 \in \R$. Again without any loss of generality we also assume $a_0 < 0 < a_1$. 

\bigskip

Consider any $a_{0} <A_{0} < A_{1} < a_{1}$ and $L_{0}, L_{1} >0$ such that $A_{0} + L_{0} < A_{1}$ and $A_{1} + L_{1} < a_{1}$.
Then define the following two probability measures
$$
\mu_{0} : = \int_{Q} g(q,\cdot)_\sharp \left(  \frac{1}{L_{0}} \mathcal{L}^{1}\llcorner_{ [A_{0},A_{0}+L_{0}] } \right) \, \qq(dq), \qquad 
\mu_{1} : = \int_{Q} g(q,\cdot)_\sharp \left( \frac{1}{L_{1}} \mathcal{L}^{1}\llcorner_{ [A_{1},A_{1}+L_{1}] } \right) \, \qq(dq).
$$
Since $g(q,\cdot)$ is an isometry one can also represent $\mu_{0}$ and $\mu_{1}$ in the following way: 
$$
\mu_{i} : = \int_{Q} \frac{1}{L_{i}}  \mathcal{H}^{1}\llcorner_{ \left\{g(q,t) \colon t \in [A_{i},A_{i}+L_{i}] \right\} } \, \qq(dq)
$$
for $i =0,1$. Both $\mu_{i}$ are absolutely continuous with respect to $\mm$ and $\mu_{i} = \r_{i} \mm$
with 
$$
\r_{i} (g(q,t)) = \frac{1}{L_{i}} h_{q}(t)^{-1}, \qquad \forall \, t \in [A_{i},A_{i}+L_{i}]. 
$$
Moreover from Lemma \ref{L:12monotone} it follows that the curve $[0,1] \ni s \mapsto \mu_{s} \in \mathcal{P}(X)$  defined by 
$$
\mu_{s} : = \int_{Q} \frac{1}{L_{s}}  \mathcal{H}^{1}\llcorner_{ \left\{g(q,t) \colon t \in [A_{s},A_{s}+L_{s}] \right\} } \, \qq(dq)
$$
where 
$$
L_{s} : = (1 - s)L_{0} + sL_{1}, \qquad A_{s} : = (1-s ) A_{0} + s A_{1}
$$
is the unique $L^{2}$-Wasserstein geodesic connecting $\mu_{0}$ to $\mu_{1}$. Again one has $\mu_{s} = \r_{s} \mm$ and can also write its density in the following way:
$$
\r_{s} (g(q,t)) = \frac{1}{L_{s}} h_{q}(t)^{-1}, \qquad \forall \, t \in [A_{s},A_{s}+L_{s}]. 
$$

{\bf Step 2.}\\
By $\CD_{loc}(K,N)$ and the essentially non-branching property one has: for $\qq$-a.e. $q \in Q_{i}$
$$
(L_{s})^{\frac{1}{N}} h_{q}( (1-s) t_{0} + s t_{1} )^{\frac{1}{N}}
	\geq 		\tau_{K,N}^{(1-s)}(t_{1}-t_{0}) (L_{0})^{\frac{1}{N}} h_{q}( t_{0} )^{\frac{1}{N}}+  \tau_{K,N}^{(s)}(t_{1}-t_{0}) (L_{1})^{\frac{1}{N}} h_{q}( t_{1} )^{\frac{1}{N}}, 
$$
for $\L^{1}$-a.e. $t_{0} \in [A_{0},A_{0} + L_{0}]$ and $t_{1}$ obtained as the image of $t_{0}$ through the monotone rearrangement of $[A_{0},A_{0}+L_{0}]$ to 
$[A_{1},A_{1}+L_{1}]$ and every $s \in [0,1]$. If $t_{0} = A_{0} + \tau L_{0}$, then $t_{1} = A_{1} + \tau L_{1}$. Also $A_{0}$ and $A_{1} +L_{1}$ should be taken close enough to 
verify the local curvature condition. 

Then we can consider the previous inequality only for $s = 1/2$ and include the explicit formula for $t_{1}$ and obtain: 
\begin{align*}
(L_{0} + L_{1})^{\frac{1}{N}} &h_{q}(A_{1/2} + \tau L_{1/2})^{\frac{1}{N}} \\
			&	\geq  
		\sigma^{(1/2)}_{K,N-1}( A_{1} - A_{0} + \tau |L_{1} - L_{0}| )^{\frac{N-1}{N}} \left\{ (L_{0})^{\frac{1}{N}} h_{q}(A_{0} + \tau L_{0})^{\frac{1}{N}} 
			+ (L_{1})^{\frac{1}{N}} h_{q}(A_{1} + \tau L_{1})^{\frac{1}{N}} \right\},
\end{align*}
for $\L^{1}$-a.e. $\tau \in [0,1]$, where we used the notation $A_{1/2}:=\frac{A_0+A_1}{2}, L_{1/2}:=\frac{L_0+L_1}{2}$. 
Now observing that the map $s \mapsto h_{q}(s)$ is continuous, the previous inequality also holds for $\tau =0$:
\begin{equation}\label{E:beforeoptimize}
(L_{0} + L_{1})^{\frac{1}{N}} h_{q}(A_{1/2} )^{\frac{1}{N}}
		\geq 
		\sigma^{(1/2)}_{K,N-1}( A_{1} - A_{0})^{\frac{N-1}{N}} 
				\left\{ (L_{0})^{\frac{1}{N}} h_{q}(A_{0})^{\frac{1}{N}} + (L_{1})^{\frac{1}{N}} h_{q}(A_{1})^{\frac{1}{N}} \right\},
\end{equation}
for all $A_{0} < A_{1}$  with $A_{0},A_{1}\in (a_0, a_1)$, all sufficiently small $L_{0}, L_{1}$ and $\qq$-a.e. $q\in Q$, 
with exceptional set depending on $A_{0},A_{1},L_{0}$ and $L_{1}$. 

Noticing that \eqref{E:beforeoptimize} depends in a continuous way on $A_{0},A_{1},L_{0}$ and $L_{1}$, it follows that there 
exists a common exceptional set $N \subset Q$ such that $\qq(N) = 0$ and for each $q \in Q\setminus N$ for all  
$A_{0},A_{1},L_{0}$ and $L_{1}$ the inequality \eqref{E:beforeoptimize} holds true.
Then one can make the following (optimal) choice 
$$
L_{0} : = L \frac{h_{q}(A_{0})^{\frac{1}{N-1}}  }{h_{q}(A_{0})^{\frac{1}{N-1}} + h_{q}(A_{1})^{\frac{1}{N-1}} }, \qquad 
L_{1} : = L \frac{h_{q}(A_{1})^{\frac{1}{N-1}}  }{h_{q}(A_{0})^{\frac{1}{N-1}} + h_{q}(A_{1})^{\frac{1}{N-1}} },
$$
for any $L > 0$ sufficiently small, and obtain that 
\begin{equation}\label{E:CDKN-1}
h_{q}(A_{1/2} )^{\frac{1}{N-1}}
		\geq 
		\sigma^{(1/2)}_{K,N-1}( A_{1} - A_{0}) 
				\left\{  h_{q}(A_{0})^{\frac{1}{N-1}} + h_{q}(A_{1})^{\frac{1}{N-1}} \right\}.
\end{equation}
Now one can observe that \eqref{E:CDKN-1} is precisely the inequality requested for $\CD^{*}_{loc}(K,N-1)$ to hold. 
As stated in Section \ref{Ss:geom}, the reduced curvature-dimension condition verifies the local-to-global property. 
In particular, see \cite[Lemma 5.1, Theorem 5.2]{cavasturm:MCP}, if a function verifies \eqref{E:CDKN-1} locally, 
then it also satisfies it globally. 
Hence $h_{q}$ also verifies the inequality requested for $\CD^{*}(K,N-1)$ to hold, i.e. for $\qq$-a.e. $q \in Q$, the density $h_{q}$ verifies \eqref{E:curvdensmm}.
\\

{\bf Step 3.}\\ 
For the case $N =1$, repeat the same construction of {\bf Step 1.} and obtain for $\qq$-a.e. $q \in Q$
$$
(L_{s}) h_{q}( (1-s) t_{0} + s t_{1} )	\geq 		(1-s) L_{0} h_{q}( t_{0} )+  s L_{1} h_{q}( t_{1} ),
$$
for any $s \in [0,1]$ and $L_{0}$ and $L_{1}$ sufficiently small. As before, we deduce for $s = 1/2$ that
$$
\frac{L_{0} + L_{1}}{2} h_{q}( A_{1/2} ) \geq  \frac{1}{2}   \left(L_{0} h_{q}( A_{0} )+  L_{1}h_{q}( A_{1} ) \right).
$$
Now taking $L_{0} = 0$ or $L_{1} = 0$, it follows that necessarily $h_{q}$ has to be constant. 
\end{proof}

Accordingly to Remark \ref{R:CDN-1},  Theorem \ref{T:CDKN-1} can be alternatively stated as follows. \\

\noindent
\emph{
If $(X,\sfd,\mm)$ is an essentially non-branching m.m.s. verifying $\CD_{loc}(K,N)$
and $\f : X \to \R$ is a 1-Lipschitz function, then the corresponding decomposition of the space 
in maximal rays $\{ X_{q}\}_{q\in Q}$ produces a disintegration $\{\mm_{q} \}_{q\in Q}$ of $\mm$ so that  for $\qq$-a.e. $q\in Q$, 
$$
\textrm{the m.m.s. }(  \dom(g(q,\cdot)), |\cdot|, h_{q} \mathcal{L}^{1}) \quad \textrm{verifies} \quad \CD(K,N).
$$
}
Accordingly,  one says that the disintegration $q \mapsto \mm_{q}$ is a $\CD(K,N)$ disintegration.

\bigskip

The disintegration obtained with $L^{1}$-Optimal Transportation is also balanced in the sense of Section \ref{Ss:balanced}.
This additional information together with what proved so far is collected in the next

\begin{theorem}[Theorem 5.1 of \cite{CM1}]\label{T:localize}
Let $(X,\sfd, \mm)$ be an essentially non-branching metric measure space verifying the $\CD_{loc}(K,N)$ condition for some $K\in \R$ and $N\in [1,\infty)$. 
Let $f : X \to \R$ be $\mm$-integrable such that $\int_{X} f\, \mm = 0$ and assume the existence 
of $x_{0} \in X$ such that $\int_{X} | f(x) |\,  \sfd(x,x_{0})\, \mm(dx)< \infty$. 
\medskip

Then the space $X$ can be written as the disjoint union of two sets $Z$ and $\mathcal{T}$ with $\mathcal{T}$ admitting a partition 
$\{ X_{q} \}_{q \in Q}$ and a corresponding disintegration of $\mm\llcorner_{\mathcal{T}}$, $\{\mm_{q} \}_{q \in Q}$ such that: 

\begin{itemize}
\item For any $\mm$-measurable set $B \subset \mathcal{T}$ it holds 
$$
\mm(B) = \int_{Q} \mm_{q}(B) \, \qq(dq), 
$$
where $\qq$ is a probability measure over $Q$ defined on the quotient $\sigma$-algebra $\mathcal{Q}$. 
\medskip
\item For $\qq$-almost every $q \in Q$, the set $X_{q}$ is a geodesic and $\mm_{q}$ is supported on it. 
Moreover $q \mapsto \mm_{q}$ is a $\CD(K,N)$ disintegration.
\medskip
\item For $\qq$-almost every $q \in Q$, it holds $\int_{X_{q}} f \, \mm_{q} = 0$ and $f = 0$ $\mm$-a.e. in $Z$.
\end{itemize}
\end{theorem}

The proof is just a collection of already proven statements. We include it for readers convenience.

\begin{proof} 
Consider 
$$
\mu_{0} : = f_{+} \mm \frac{1}{\int f_{+}\mm}, \qquad \mu_{1} : = f_{-}\mm \frac{1}{\int f_{-}\mm},
$$ 
where $f_{\pm}$ stands for the positive and negative part of $f$, respectively. 
From the summability assumption on $f$ it follows the existence of $\f : X \to \R$, $1$-Lipschitz Kantorovich potential for the couple of marginal probability $\mu_{0}, \mu_{1}$. 
Since the m.m.s. $(X,\sfd,\mm)$ is essentially non-branching, the transport  set $\T$ is partitioned by the rays: 
$$
\mm_{\T} = \int_{Q} \mm_{q}\, \qq(dq),\qquad \mm_{q}(X_{q}) = 1, \quad \qq-\textrm{a.e. } q \in Q;
$$
moreover $(X,\sfd,\mm)$ verifies $\CD_{loc}$ and therefore Theorem \ref{T:CDKN-1} implies that $q \mapsto \mm_{q}$ is a $\CD(K,N)$ disintegration.
Lemma \ref{L:balancing} implies that 
$$
\int_{X_{q}} f(x) \, \mm_{q}(dx) = 0.
$$
To conclude moreover note that in $X \setminus \T$ necessarily $f$ has to be zero.
Take indeed any $B \subset X\setminus \T$ compact with $\mm(B) > 0$ and assume $f \neq 0$ over $B$. Then possibly taking a subset, we can assume $f > 0$ over $B$ 
and therefore $\mu_{0}(B) > 0$. Since 
$$
\mu_{0} = \int_{Q} \mu_{0\,q} \qq(dq), \qquad \mu_{0\,q} (X_{q}) = 1,  
$$
necessarily $B$ cannot be a subset of $X\setminus \T$ yielding a contradiction. All the claims are proved.
\end{proof}

\bigskip

 
\section{Applications}\label{S:application}

Here we will collect some applications of the results proved so far, in particular of Proposition \ref{P:nonatoms} and Theorem \ref{T:CDKN-1}

\subsection{Solution of the Monge problem}
 
Here we review how regularity of conditional probabilities of the one-dimensional disintegration studied so far permits to construct a solution to the Monge 
problem. In particular we will see how Proposition \ref{P:nonatoms} allows to construct an optimal map $T$. 
As the plan is to use the one-dimensional reduction, first we recall the one dimensional result for the Monge problem \cite{Vil}.

\begin{theorem}
\label{T:oneDmonge}
Let $\mu_{0}, \mu_{1}$ be probability measures on $\erre$, $\mu_{0}$ with no atoms, and let
$$
H(s) := \mu_{0}((-\infty,s)), \quad F(t) := \mu_{1}((-\infty,t)),
$$
be the left-continuous distribution functions of $\mu_{0}$ and $\mu_{1}$ respectively. Then the following holds.
\begin{enumerate}
\item The non decreasing function $T : \erre \to \erre \cup [-\infty,+\infty)$ defined by
$$
T(s) := \sup \big\{ t \in \erre : F(t) \leq H(s) \big\}
$$
maps $\mu_{0}$ to $\mu_{1}$. Moreover any other non decreasing map $T'$ such that $T'_\sharp \mu_{0} = \mu_{1}$ coincides 
with $T$ on the support of $\mu_{0}$ up to a countable set.
\item If $\phi : [0,+\infty] \to \erre$ is non decreasing and convex, then $T$ is an optimal transport relative to the cost $c(s,t) = \phi(|s-t|)$. 
Moreover $T$ is the unique optimal transference map if $\phi$ is strictly convex.
\end{enumerate}
\end{theorem}

\begin{theorem}[Theorem 6.2 of \cite{biacava:streconv}]\label{T:mongeff}
Let $(X,\sfd,\mm)$ be a  non-branching metric measure space and consider $\mu_{0}, \mu_{1} \in \mathcal{P}(X)$ with finite first moment. 
Assume the existence of a Kantorovich potential $\f$ such that the associated transport set $\T$ verifies \emph{\ref{A:1}}.
Assume $\mu_{0} \ll \mm$. 

Then there exists a Borel map $T: X \to X$ such that 
$$
\int_{X} \sfd(x,T(X)) \, \mu_{0} (dx) = \min_{\pi \in \Pi(\mu_{0},\mu_{1})} \int_{X\times X} \sfd(x,y) \, \pi(dxdy).
$$
\end{theorem}

Theorem \ref{T:mongeff} was presented in \cite{biacava:streconv} assuming the space to be non-branching, while here we assume essentially non-branching.

\begin{proof}

{\bf Step 1.} One dimensional reduction of $\mu_{0}$. \\
Let $\f : X \to \R$ be the Kantorovich potential from the assumptions and $\T$ the corresponding transport set. Accordingly 
$$
\mm\llcorner_{\T} = \int_{Q} \mm_{q}\,\qq(dq), 
$$
with $\mm_{q}(X_{q}) = 1$ for $\qq$-a.e. $q \in Q$. Moreover from \ref{A:1} for $\qq$-a.e. $q \in Q$ 
the conditional $\mm_{q}$ has no atoms, i.e. $\mm_{q}(\{z \}) = 0$ for all $z \in X$. 
From Lemma \ref{L:mapoutside}, we can assume that $\mu_{0}(\T_{e}) = \mu_{1}(\T_{e}) = 1$. Since $\mu_{0} = \r_{0} \mm$, with $\r_{0} : X \to [0,\infty)$,
from Theorem \ref{T:equivalence} we have $\mu_{0}(\T) = 1$. Hence
$$
\mu_{0} =  \int_{Q} \r_{0} \mm_{q} \, \qq(dq) =  \int_{Q}  \mu_{0 \, q} \, \qq_{0}(dq),  \qquad 
\mu_{0\, q} : = \r_{0} \mm_{q} \left(\int_{X} \r_{0}(x) \mm_{q}(dx) \right)^{-1},
$$
and $\qq_{0} = \QQ_{\sharp} \mu_{0}$. In particular $\mu_{0,\, q}$ has no atoms and $\mu_{0\,q}(X_{q}) = 1$.
\medskip

{\bf Step 2.} One dimensional reduction of $\mu_{1}$. \\
As we are not making any assumption on $\mu_{1}$ we cannot exclude that $\mu_{1}(\T_{e} \setminus \T) >0$ 
and therefore to localize $\mu_{1}$ one cannot proceed as for $\mu_{0}$.
Consider therefore an optimal transport plan $\pi$ with $\pi(\Gamma) = 1$. Since $\pi (\T \times \T_{e}) = 1$ and a partition of $\T$ is given, 
we can consider the following family of sets $\{ X_{q} \times \T_{e}\}_{q\in Q}$ as a partition of $\T \times \T_{e}$; 
note indeed that $X_{q} \times \T_{e} \cap X_{q'} \cap \T_{e} = \emptyset$ as soon as $q \neq q'$.
The domain of the quotient map $\QQ : \T \to Q$ can be trivially extended to $\T \times \T_{e}$ by saying that $\QQ(x,z) = \QQ(x)$ 
and observe that 
$$
\QQ_{\sharp} \,\pi(I) =  \pi \left( \QQ^{-1} (I) \right) =\pi \left( \QQ^{-1}(I) \times \T_{e} \right) =\mu_{0}(\QQ^{-1}(I)) =\qq_{0}(I).
$$
In particular this implies that 
$$
\pi = \int_{Q} \pi_{q} \, \qq_{0}(dq), \qquad \pi_{q} (X_{q} \times \T_{e}) = 1, \quad \textrm{for }  \qq_{0}\textrm{-a.e. } q \in Q. 
$$
Then applying the projection 
$$
\mu_{0} = P_{1\,\sharp} \pi = \int_{Q} P_{1\,\sharp}(\pi_{q}) \, \qq_{0}(dq),
$$
and by uniqueness of disintegration $P_{1\,\sharp}(\pi_{q}) = \mu_{0\,\qq}$ for $\qq_{0}$-a.e. $q\in Q$. Then we can find a localization of $\mu_{1}$ as follows:
$$
\mu_{1} = P_{2\,\sharp} \pi = \int_{Q} P_{2\,\sharp}(\pi_{q}) \, \qq_{0}(dq) = \int_{Q} \mu_{1\,q} \, \qq_{0}(dq),
$$
where by definition we posed $\mu_{1\, q} : = P_{2\,\sharp}(\pi_{q})$ and by construction $\mu_{1\, q} (X_{q}) = \mu_{0\,q}(X_{q}) = 1$.
\medskip

{\bf Step 3.} Solution to the Monge problem.\\
For each $q \in Q$ consider the distribution functions
$$
H(q,t) := \mu_{0\,q}((-\infty,t)), \quad F(q, t) := \mu_{1\,q }((-\infty,t)),
$$
where for ease of notation $\mu_{i\,q} = g(q,\cdot)^{-1}_{\sharp} \mu_{i\,q}$ for $i = 0,1$.
Then define $\hat T$, as Theorem \ref{T:oneDmonge} suggests, by
$$
\hat T(q,s) := \Big(q, \sup \big\{ t : F(q,t) \leq H(q,s) \big\} \Big). 
$$
Note that since $H$ is continuous ($\mu_{0\,q}$ has no atoms), the map $s \mapsto \hat T(q,s)$ is well-defined.
Then define the transport map $T: \T \to X$ as $g\circ \hat T \circ g^{-1}$.
It is fairly easy to observe that 
$$
T_{\sharp} \,\mu_{0} = \int_{Q} \left(g\circ \hat T \circ g^{-1}\right)_{\sharp} \mu_{0\, q}\, \qq_{0}(dq) = \int_{Q} \mu_{1\, q} \,\qq_{0}(dq) = \mu_{1};
$$
moreover $(x,T(x)) \in \Gamma$ and therefore the graph of $T$ is $\sfd$-cyclically monotone and therefore the map $T$ is optimal.
Extend $T$ to $X$ as the identity. 

It remains to show that it is Borel. First observe that, possibly taking a compact subset of $Q$
the map $q \mapsto (\mu_{0\,q},\mu_{1\,q})$ can be assumed to be weakly continuity; it follows that the maps
$$
\dom(g) \ni (q,t) \mapsto H(q,t) := \mu_{0\,q}((-\infty,t)), \ \ (q,t) \mapsto F(q,t) := \mu_{1\,q}((-\infty,t))
$$
are lower semicontinuous.  Then for $A$ Borel,
$$
\hat T^{-1}(A \times [t,+\infty)) = \big\{ (q,s) : q \in A, H(q,s) \geq F(q,t) \big\} \in \mathcal{B}(Q \times \R),
$$
and therefore the same applies for $T$.
\end{proof}

If $(X,\sfd,\mm)$ verifies $\MCP$ then it also verifies \ref{A:1}, see Proposition \ref{P:MCP-1}. So we have the following

\begin{corollary}[Corollary 9.6 of \cite{biacava:streconv}]\label{C:MCP-Monge}
Let $(X,\sfd,\mm)$ be a  non-branching metric measure space verifying $\MCP(K,N)$. Let $\mu_{0}$ and $\mu_{1}$ be probability measures
with finite first moment and $\mu_{0}\ll \mm$. Then there exists a Borel optimal transport map $T: X \to X$ solution to the Monge problem.
\end{corollary}

Corollary \ref{C:MCP-Monge} in particular implies the existence of solutions to the Monge problem in the Heisenberg group when $\mu_{0}$ is assumed 
to be absolutely continuous with respect to the left-invariant Haar measure. 

\begin{theorem}[Monge problem in the Heisenberg group]
Consider $(\HH^n, \sfd_c,\L^{2n+1} )$, the $n$-dimensional Heisenberg group endowed with the Carnot-Carath\'eodory distance $\sfd_{c}$ and 
the $(2n+1)$-Lebesgue measure that coincide with the Haar measure on $(\HH^n, \sfd_c)$ under the identification $\HH^n\simeq \R^{2n+1}$.
Let $\mu_{0}$ and $\mu_{1}$ be two probability measures with finite first moment and $\mu_{0}\ll \L^{2n+1}$.  
Then there exists a Borel optimal transport map $T: X \to X$ solution to the Monge problem.
\end{theorem}

\begin{remark}
The techniques used so far were successfully used also to threat the more general case of infinite dimensional spaces with curvature bound, see \cite{cava:Wiener} where 
the existence of solutions for the Monge minimization problem in the Wiener space is proved. 
Note that the material presented in the previous sections can be obtained also without assuming the existence of a $1$-Lipschitz Kantorovich potential (e.g. the Wiener space);  
the decomposition of the space in geodesics and the associated disintegration of the reference measures can be obtained starting from a generic $\sfd$-cyclically monotone set. 
For all the details see \cite{biacava:streconv}.
\end{remark}

\bigskip
\subsection{Isoperimetric inequality}

We now turn to the second main application of techniques reviewed so far, 
the L\'evy-Gromov isoperimetric inequality in singular spaces. The results of this section are taken from \cite{CM1,CM2}.

\begin{theorem}[Theorem 1.2 of \cite{CM1}]\label{T:iso}
Let $(X,\sfd,\mm)$ be a metric measure space with $\mm(X)=1$, verifying  the essentially non-branching property and $\CD_{loc}(K,N)$ for some $K\in \R,N \in [1,\infty)$.
Let $D$ be the diameter of $X$, possibly assuming the value $\infty$.
\medskip

Then for every $v\in [0,1]$, 
$$
\cI_{(X,\sfd,\mm)}(v) \ \geq \ \cI_{K,N,D}(v), 
$$
where $\cI_{K,N,D}$ is the model isoperimetric profile defined in \eqref{defcI}.
\end{theorem}

\begin{proof}
First of all we can assume $D<\infty$ and therefore $\mm \in \mathcal{P}_{2}(X)$: indeed from the Bonnet-Myers Theorem if $K>0$ then $D<\infty$, and if $K\leq 0$ and $D=\infty$ then the model isoperimetric profile \eqref{defcI} trivializes, i.e. $\cI_{K,N,\infty}\equiv 0$ for $K\leq 0$.

For $v=0,1$ one can take as competitor the empty set and the whole space respectively, so it trivially holds 
$$
\cI_{(X,\sfd,\mm)}(0)=\cI_{(X,\sfd,\mm)}(1)= \cI_{K,N,D}(0)=\cI_{K,N,D}(1)=0.
$$
Fix then $v\in(0,1)$ and let $A\subset X$  be an arbitrary Borel subset of $X$ such that $\mm(A)=v$.
Consider the $\mm$-measurable function $f(x) : = \chi_{A}(x)  - v$ and notice that  $\int_{X} f \, \mm = 0$. 
Thus $f$ verifies the  hypothesis of Theorem \ref{T:localize} and noticing that $f$ is never null, 
we can decompose $X = Y \cup \mathcal{T}$ with 
$$
\mm(Y)=0, \qquad   \mm\llcorner_{\mathcal{T}} = \int_{Q} \mm_{q}\, \qq(dq), 
$$
with $\mm_{q} = g(q,\cdot)_\sharp \left( h_{q} \cdot \mathcal{L}^{1}\right)$; 
moreover,  for $\qq$-a.e. $q \in Q$,  the density $h_{q}$ verifies \eqref{E:curvdensmm}  and 
$$
\int_{X} f(z) \, \mm_{q}(dz) =  \int_{\dom(g(q,\cdot))} f(g(q,t)) \cdot h_{q}(t) \, \mathcal{L}^{1}(dt) = 0.
$$
Therefore 
\begin{equation}\label{eq:volhq}
v=\mm_{q} ( A \cap \{ g(q,t) : t\in \R \} ) = (h_{q}\mathcal{L}^1) (g(q,\cdot)^{-1}(A)), \quad \text{ for $\qq$-a.e. $q \in Q$}. 
\end{equation}
For every $\ve>0$ we  then have  
\begin{align*}
\frac{\mm(A^\ve)-\mm(A)}{\ve} 	&~  =  \frac{1}{\ve} \int_{\mathcal{T}} \chi_{A^\ve\setminus A} \,\mm(dx) 
									=  \frac{1}{\ve} \int_{Q} \left( \int_{X}   \chi_{A^\ve\setminus A} \, \mm_{q} (dx) \right)\, \qq(dq) \crcr
						&~ =    \int_{Q} \frac{1}{\ve} \left( \int_{\dom(g(q,\cdot))}  \chi_{A^\ve\setminus A} \,h_{q}(t) \, \mathcal{L}^{1}(dt) \right)\, \qq(dq) \crcr
						&~ =    \int_{Q} \left(  \frac{(h_{q}\mathcal{L}^1)(g(q,\cdot)^{-1}(A^\ve))-  (h_{q}\mathcal{L}^1)(g(q,\cdot)^{-1}(A))}{\ve}  \right)\, \qq(dq) \crcr
						&~ \geq    \int_{Q} \left(  \frac{(h_{q}\mathcal{L}^1)((g(q,\cdot)^{-1}(A))^\ve)-  (h_{q}\mathcal{L}^1)(g(q,\cdot)^{-1}(A))}{\ve}  \right)\, \qq(dq), \crcr
\end{align*}
where the last inequality is given by the inclusion $ (g(q,\cdot)^{-1}(A))^\ve \cap \supp(h_q) \subset  g(q,\cdot)^{-1}(A^\ve)$. \\
Recalling \eqref{eq:volhq} together with $h_{q}\mathcal{L}^1\in \mathcal{F}^{s}_{K,N,D}$, by Fatou's Lemma we get
\begin{align*}
\mm^+(A)		&~ = 	\liminf_{\ve\downarrow 0} \frac{\mm(A^\ve)-\mm(A)}{\ve} \crcr
			&~\geq 	\int_{Q}  \left(  \liminf_{\ve\downarrow 0} \frac{(h_{q}\mathcal{L}^1)((g(q,\cdot)^{-1}(A))^\ve) -  
											(h_{q}\mathcal{L}^1)(g(q,\cdot)^{-1}(A))}{\ve}  \right)\, \qq(dq) \crcr
			&~ =   	\int_{Q} \left( (h_{q}\mathcal{L}^1)^+(g(q,\cdot)^{-1}(A))  \right)\, \qq(dq) \crcr
			&~ \geq  	\int_{Q}  \cI^s_{K,N,D} (v) \, \qq(dq)  \crcr
			&~ =  	\cI_{K,N,D} (v),
\end{align*}
where in the last equality we used Theorem \ref{thm:I=Is}. 
\end{proof}

From the definition of $\mathcal{I}_{K,N,D}$, see \eqref{defcI}, and the smooth results of E. Milman in \cite{Mil},  the estimates proved in Theorem \ref{T:iso} are sharp.

Furthermore, 1-dimensional localization technique permits to obtain rigidity in the following sense: 
if for some $v \in (0,1)$ it holds $\cI_{(X,\sfd,\mm)}(v)= \cI_{K,N,\pi}(v)$, then $(X,\sfd,\mm)$ is a spherical suspension. 
It is worth underlining that to obtain such a result $(X,\sfd,\mm)$ is assumed to be in the more regular class of $\RCD$-spaces.

Even more, one can prove an almost rigidity statement: if $(X,\sfd,\mm)$ is an $\RCD^*(K,N)$ space such that   $\cI_{(X,\sfd,\mm)}(v)$ is close to $\cI_{K,N,\pi}(v)$ 
for some $v \in (0,1)$,  this force $X$ to be close, in the measure-Gromov-Hausdorff distance, to a spherical suspension. What follows is Corollary 1.6 of \cite{CM1}.

\begin{theorem}[Almost equality in L\'evy-Gromov implies mGH-closeness to a spherical suspension] \label{cor:AlmRig}
For every $N\in [2, \infty) $, $v \in (0,1)$, $\ve>0$ there exists $\bar{\delta}=\bar{\delta}(N,v,\ve)>0$ such that the following hold. For every  $\delta \in [0, \bar{\delta}]$, if  $(X,\sfd,\mm)$ is an $\RCD^*(N-1-\delta,N+\delta)$ space satisfying 
$$
\cI_{(X,\sfd,\mm)}(v)\leq \cI_{N-1,N,\pi}(v)+\delta, 
$$
then  there exists an $\RCD^*(N-2,N-1)$ space $(Y, \sfd_Y, \mm_Y)$ with $\mm_Y(Y)=1$ such that 
$$
\sfd_{mGH}(X, [0,\pi] \times_{\sin}^{N-1} Y) \leq \ve. 
$$
\end{theorem}

We refer to \cite{CM1} for the precise rigidity statement (Theorem 1.4, \cite{CM1}) and for the proof of Theorem 1.4 and Corollary 1.6 of \cite{CM1}.
See also \cite{CM1} for the precise definition of spherical suspension.
We conclude by recalling that 1-dimensional localization was used also in \cite{CM2} to obtain sharp version of several functional inequalities 
(e.g. Brunn-Minkowski, spectral gap, Log-Sobolev etc.) in the class of $\CD(K,N)$-spaces. See \cite{CM2} for details.

\end{document}